\documentclass[12pt,reqno]{amsart}

\usepackage{comment}
\usepackage{color}
\usepackage{amsmath}
\usepackage{amsthm}
\usepackage{amssymb}
\usepackage{graphicx}
\usepackage{enumerate}
\usepackage{amsfonts}
\usepackage{mathrsfs}
\usepackage{parskip}
\usepackage{mathdots}
\usepackage{color}

\numberwithin{equation}{section}
\theoremstyle{plain}
\newtheorem{Proposition}[equation]{Proposition}
\newtheorem{Corollary}[equation]{Corollary}
\newtheorem*{Corollary*}{Corollary}
\newtheorem{Theorem}[equation]{Theorem}
\newtheorem*{Theorem*}{Theorem}
\newtheorem{Lemma}[equation]{Lemma}
\theoremstyle{definition}
\newtheorem{Definition}[equation]{Definition}

\newtheorem{Remark}[equation]{Remark}

\allowdisplaybreaks

\usepackage{tikz}
\usetikzlibrary{decorations.pathreplacing}

\usepackage{enumitem}
\setlist[enumerate]{leftmargin=*}
\setlist[itemize]{leftmargin=*}

\setlist[enumerate,1]{label=(\alph*),font=\upshape}

\setlist[enumerate,2]{label=(\roman*),font=\upshape}

\newcommand{\R}{\mathbb{R}}
\newcommand{\C}{\mathbb{C}}

\newcommand{\D}{\mathbb{D}}
\newcommand{\N}{\mathbb{N}}

\renewcommand{\leq}{\leqslant}
\renewcommand{\geq}{\geqslant}
\renewcommand{\subset}{\subseteq}


\begin{document}
\title[Zeros in $\ell_A^p$]{Zeros of optimal polynomial approximants in $\ell^p_{A}$}
\author[Cheng]{Raymond Cheng}
\address{Department of Mathematics and Statistics, Old Dominion University, Norfolk, VA 23529, USA. } \email{rcheng@odu.edu}
\author[Ross]{William T. Ross}
\address{Department of Mathematics and Computer Science, University of Richmond, Richmond, VA 23173, USA. } \email{wross@richmond.edu}
\author[Seco]{Daniel Seco}
\address{Universidad Carlos III de Madrid and Instituto de Ciencias Matem\'aticas, Departamento de Matem\'aticas, Avenida de la Universidad 30, 28911 Legan\'es (Madrid), Spain.} \email{dseco@math.uc3m.es}

\date{\today}

\begin{abstract}
The study of inner and cyclic functions in $\ell^p_A$ spaces requires a better understanding of the zeros of the so-called optimal polynomial approximants. 
 We determine that a point of the complex plane is the zero of an optimal polynomial approximant for some element of $\ell^p_A$ if and only if it lies outside of a closed disk (centered at the origin) of a particular radius which depends on the value of $p$. We find the value of this radius for $p\neq 2$. In addition, for each positive integer $d$ there is
 a polynomial $f_d$ of degree at most $d$ that minimizes the modulus of 
the root of its optimal linear polynomial approximant.  We develop a method for finding these extremal functions $f_d$ and discuss their properties.  The method involves the Lagrange multiplier method and a resulting dynamical system.
\end{abstract}

\maketitle

\section{Introduction}\label{Intro}

For  $1 < p < \infty$, the space $\ell_{A}^{p}$ is the set of power series $f$ whose Taylor coefficients $(a_n)_{n = 0}^{\infty}$ belong to the standard sequence space $\ell^p(\N_{0})$.  This paper concerns the zeros of optimal polynomial approximants in $\ell^{p}_{A}$.   Every  $f \in \ell^{p}_{A}$ is an analytic function on the open unit disk $\D = \{z: |z| < 1\}$ and, endowed with the norm $\|f\|_{p} := \|(a_n)_{n = 0}^{\infty}\|_{\ell^p}$, $\ell^{p}_{A}$ is a Banach space. When $p = 2$, $\ell^{2}_{A}$ is a Hilbert space with inner product $\langle f, g\rangle := \sum_{n = 0}^{\infty} a_n \overline{b_n}$, where $g(z) = \sum_{n=0}^{\infty} b_n z^n$. In fact,  $\ell^{2}_{A}$ is the well-studied Hardy space. 

Our motivation stems from the ongoing exploration of  the invariant subspaces and the cyclic vectors for the shift operator 
$(S f)(z) = z f(z)$ on $\ell^p_{A}$. By {\em invariant subspace}, we mean a closed subspace $\mathscr{M} \subset \ell^{p}_{A}$ for which $S \mathscr{M} \subset \mathscr{M}$. By {\em cyclic vector,} we mean an $f \in \ell^{p}_{A}$ for which $\overline{\operatorname{span}} \{S^{n} f: n \in \N_{0}\} = \ell^{p}_{A}.$

Beurling's seminal paper \cite{MR27954} determined both the cyclic vectors and invariant subspaces of $\ell^{2}_{A}$. Indeed, $\mathscr{M}$ is an invariant subspace of $\ell^{2}_{A}$ if and only if $\mathscr{M} = \Theta \ell^{2}_{A} = \{\Theta f: f \in \ell^{p}_{A}\}$ for some inner function $\Theta$. A vector $f \in \ell^{2}_{A}$ is cyclic if and only if $f$ is an outer function. Both inner functions and outer functions have specific formulas through classical theorems of Nevanlinna and Riesz \cite{Dur, Gar}. Beurling pondered the invariant subspaces and cyclic vectors for $p \not = 2$, but the question remains very much unresolved. In fact, when $p > 2$, the structure of the invariant subspaces of $\ell^{p}_{A}$ is very complicated \cite{MR1878629}.

In this paper, we focus on one of the many obstacles to understanding the invariant subspaces and cyclic vectors for $\ell^{p}_{A}$ when $p \not = 2$ -- the ``extra zeros'' that appear in optimal polynomial approximants and in the related $p$-inner functions. 
To better explain what we mean here, let $\mathscr{P}_n$ be the set of polynomials of degree at most $n$, and let $\mathscr{P} := \bigcup_{n \geq 0}\mathscr{P}_n$.
Observe that $f \in \ell^{p}_{A}$ is cyclic if and only if there is a sequence $(p_n)_{n = 1}^{\infty}$ in $\mathscr{P}$ such that $\|1 - p_{n} f\|_{p} \to 0$. Various papers \cite{RemBen, BenRMI} have discussed \emph{optimal polynomial approximants}. These are polynomials $p_{n,f}$ defined as follows: For a non-constant $f \in \ell^{p}_{A}$ and $n \in \N $, let $p_{n, f}$  be the unique $p_{n, f} \in \mathscr{P}_n$ such that
$$\|1 - p_{n, f} f\|_{p} = \operatorname{dist}(1, f \mathscr{P}_{n}).$$ When $1 < p < \infty$, the uniform convexity of $\ell^{p}_{A}$ ensures the uniqueness of $p_{n, f}$.  When $p = 2$, it is known that the zeros of $p_{n, f}$ lie outside $\overline{\D}$, the closure of $\D$ \cite{BenJLMS} (see also Proposition \ref{prop311}). Perhaps one might argue this is the way it should be since the  polynomials that optimally attempt to approximate $1/f$ should not carry any zeros in $\D$.   When $p \neq 2$, however, we will see that an optimal polynomial approximant $p_{n, f}$ might have zeros inside $\overline{\D}$.  This stands in the way of it truly being a close approximation of $1/f$. 

These extra zeros also appear in several ways when trying to extend Beurling's treatment of the invariant subspaces  for the $p = 2$ case. One of the innovative tricks of Beurling was to show that if $f \in \ell^{2}_{A}$, then the unique solution $\widehat{f}$ to 
$$\|f - \widehat{f}\|_{2} = \inf_{\phi \in \mathscr{P}} \|f(z) - z \phi(z) f(z)\|_{2}$$
has the property that $f - \widehat{f}$ is an inner function (up to a multiplicative constant) and that $f$ and $f - \widehat{f}$ generate the same invariant subspace. This was taken up with much success in various other function spaces such as the Bergman and Dirichlet spaces in \cite{MR1440934, MR1259923}. What complicates matters when $p \neq 2$ is that the corresponding function $f - \widehat{f}$ might have a zero in $\D$ that is not a zero of $f$ (this is the origin of the term ``extra zero'').  Thus, unlike when $p = 2$, the functions $f$ and $f - \widehat{f}$ do not generate the same invariant subspace.  To see a concrete example of how an extra zero can arise when $p = \frac{4}{3}$, we refer the reader to \cite{CD}.

To measure the extent to which these extra zeros of $p_{n, f}$ occur, we define the following set $\Omega_p \subset \C$. 

\begin{Definition}
For fixed $1 < p < \infty$, a point $z_0 \in \C$ belongs to $\Omega_p$ if there exists an $f \in \ell^{p}_{A}$ with $f(0) \not = 0$ and an $n \in \N$ such that $p_{n, f}(z_0) = 0$. In other words $z_0 \in \Omega_p$ when $z_0$ is the zero of the optimal polynomial approximant for some $f \in \ell^{p}_{A}$ with $f(0) \not= 0$. 
\end{Definition}


The assumption that $f(0) \neq 0$ is equivalent to $p_{n,f}$ not being identically $0$ (Proposition \ref{Rnozerosisi2}). When $p = 2$, previous work from  \cite{BenJLMS} shows that $\Omega_2 = \C \backslash \overline{\D}$ (see also Proposition \ref{prop311}). In particular, no $p_{n, f}$ for any non-constant $f \in \ell^{2}_{A}$ has a zero in $\D$. Our main result is the following. 

\begin{Theorem}\label{MAIN}
For $1 < p < \infty$, with $p \not = 2$, there exists a $\tau_p \in (1,2)$ such that $\Omega_{p} = \C \backslash \frac{1}{\tau_p} \overline{\D}$.
\end{Theorem}

Thus there is a striking difference between the $p = 2$ case, where $\Omega_p \cap \D = \varnothing$, and the $p \not = 2$ case, where $\Omega_p \cap \D \not = \varnothing$. In the latter case, extra zeros for optimal polynomial approximants exist in abundance. 

The proof of Theorem \ref{MAIN} involves the following steps to ``fill out'' $\Omega_p$: Lemma \ref{easy1} says that $z_0 \in \Omega_p$ if and only if there is some $f \in \ell^{p}_{A}$ with $f(0) \not = 0$ such that the optimal {\em linear} approximant $p_{1, f}$ vanishes at $z_0$, thus reducing the complexity of the problem. Proposition \ref{76y78uijsdfg} shows that 
$\C \backslash\overline{\D} \subset \Omega_{p}.$
Furthermore, Proposition \ref{1hHyyg667trRR} yields that $\Omega_p$ is rotationally symmetric, and Corollary \ref{condforextzer2}, that it is path connected.  At that point, we know that $\Omega_p$ is the complement of a disk. Two things remain to be shown: $\tau_{p} \in (1, 2)$ and $\Omega_p$ is open. 

Birkhoff-James orthogonality, and the associated Pythagorean inequalities (Lemma \ref{Pytha}), have been successfully used in \cite{ChengRoss15, Chengetal1, Chengetal2} to study $\ell^p_A$ and related spaces.  Here, these concepts are applied in Proposition \ref{emptydisks} to produce radii $\tfrac{1}{2} < r_p < 1$ such that 
$\Omega_{p} \cap r_p \D = \varnothing.$
In particular, this proves that 
$\Omega_{p} \cap \textstyle{\frac{1}{2}} \overline{\D} = \varnothing$ for all $1 < p < \infty$.
In Section \ref{Sect4} we provide  specific examples which  show that when $p \neq 2$ there are functions $f \in \ell^{p}_{A}$ (in fact polynomials) and $n \in \N$ for which the  optimal approximants $p_{n, f}$ have zeros in $\D$. This proves that  $\tau_p > 1$ and hence $\Omega_{p} \cap \D \not = \varnothing$. 

Trying to determine $\tau_p$ exactly leads to the most difficult part of this paper. We reduce the problem of determining $\tau_p$ in several ways, and thanks to the Lagrange multiplier method, we show it is equivalent to finding polynomials whose coefficients give solutions to a set of nonlinear recurrence relations. The behavior of these solutions as the degree of the polynomials grows will determine  the value of $\tau_p$. In the penultimate section of this paper, we use a dynamical systems approach to solve the recurrence relations mentioned above and  show that the limit of the solutions is not attained by a function in $\ell^p_A$. 
Thus, $1/\tau_p$ is not the zero of  an optimal approximant for any function. 
At that point we know that $\Omega_p$ is open and the proof of Theorem \ref{MAIN} is complete. Although we cannot compute $\tau_p$ in closed form, in Theorem \ref{lastth} we describe it implicitly  and provide numerical estimates for it.

As a byproduct of the proof of Theorem \ref{MAIN}, we provide a  connection between zeros of $p$-inner functions (i.e., those that arise as a metric co-projection $f - \widehat{f}$) and those of the optimal polynomial approximants (Theorem \ref{invsubs}).

\section{Notation}\label{Sect2}

Let us begin with  some basic notation. Let $\D = \{z \in \C: |z| < 1\}$ be the open unit disk in the complex plane $\C$, and $\operatorname{Hol}(\D)$  the analytic functions on $\D$. We set $\N = \{1, 2, 3 \ldots\}$ and $\N_{0} = \N \cup \{0\}$. Let $\mathscr{P} = \C[z]$ denote the vector space of all polynomials in the complex variable $z$ with complex coefficients, and for $d \in \N_0$, let $\mathscr{P}_d$ denote the vector space of polynomials of degree at most $d$.  A reference for the material below is the book \cite{CMR}.

\begin{Definition}
For $1 < p < \infty$, let  $\ell^{p}_{A}$ be the set of power series 
$f(z)=\sum_{k = 0}^{\infty} a_k z^k$
for which
\begin{equation}\label{eqn2}
\|f\|_p  :=  \left(\sum_{k = 0}^{\infty} |a_k|^p \right)^{\frac{1}{p}} < \infty.
\end{equation}
\end{Definition}
One can show that $\ell^{p}_{A} \subset \operatorname{Hol}(\D)$ and that $\ell^{p}_{A}$ is a uniformly convex Banach space. 

The uniform convexity of $\ell^{p}_{A}$ guarantees that for a closed subspace $\mathscr{V} \subset \ell^{p}_{A}$ and $f \in \ell^{p}_{A}$,  there is a unique $v_{f} \in \mathscr{V}$ such that 
\begin{equation}\label{0iuyghjkl}
\|f - v_f\|_{p} = \operatorname{dist}(f, \mathscr{V})  :=  \inf_{v \in \mathscr{V}} \|v - f\|_{p}.
\end{equation}
We say that $v_f$ is the {\em metric projection} of $f$ onto $\mathscr{V}$.
When $p = 2$, $\ell^{2}_{A}$ is a Hilbert space and the map $f \mapsto v_{f}$ is the orthogonal projection of $\ell^{2}_{A}$ onto $\mathscr{V}$ (and hence linear). When $p \not = 2$, the map $f \mapsto v_{f}$ is in general nonlinear.  However, it is continuous.

  Since the metric projection $f \mapsto v_{f}$ on $\ell^{p}_{A}$ is at the core of this paper, we omit the cases $p = 1$ and $p = \infty$ from our discussion. This is because the spaces $\ell^{1}_{A}$ and $\ell^{\infty}_{A}$ are not uniformly convex, and hence the closest point to $\mathscr{V}$ is generally not attained; furthermore, when attained, the nearest point need not be unique. In fact, this failure of uniqueness is relevant to the study of optimal polynomial approximants \cite{SeTe}. 

When dealing with approximation theory in Banach spaces without a natural concept of orthogonality, a useful substitute is Birkhoff-James orthogonality. This generalization of Hilbert space orthogonality is valid in arbitrary normed spaces and has proven to be useful in exploring invariant subspaces and cyclicity  in various Banach spaces of analytic functions \cite{Chengetal1, Chengetal2, SeTe}.

\begin{Definition}\label{BJ_defn}
For $f, g \in \ell^{p}_{A}$ we say that $f$ is \emph{Birkhoff-James orthogonal}  to $g$ if
$$\|f +\alpha g\|_p \geq \|f\|_p \quad \mbox{for all $\alpha \in \C$},$$
and in this case we write $f \perp_{p} g$. 
\end{Definition}

When $p = 2$, Birkhoff-James orthogonality coincides with orthogonality in the usual Hilbert space sense, since the definition above says that $0$ is the closest point in the subspace $\C g$ to $f$. The ordering in the relation $f \perp_{p} g$ is important since, when $p \not = 2$, it is possible that $f \perp_{p} g$ but $g \not \perp_{p} f$. There is the following tangible condition for Birkhoff-James orthogonality involving a semi-inner product on $\ell^{p}_{A}$.

 \begin{Lemma}\label{BJsemi}
If $f(z) = \sum_{k = 0}^{\infty} a_k z^k$ and  $g(z) = \sum_{k = 0}^{\infty}b_k  z^k$ belong to  $\ell^{p}_{A}$, then
 \begin{equation}\label{eqn3} 
f  \perp_{p} g \iff \sum_{k  = 0}^{\infty} |a_k|^{p-2}\overline{a_k}b_k = 0,
\end{equation}
where, in the above sum, we understand any occurrence of ``$|0|^{p - 2} 0$'' as zero.
\end{Lemma}

This lemma implies that the relation $f \perp_{p} g$ is linear in the second slot; that is, if $f \perp_p g$ and $f \perp_p h$, then for any $\alpha$, $\beta \in \C$, $f \perp_{p} (\alpha g + \beta h)$. 
This is reminiscent of an orthogonality relation in $\ell^2$, and even more so after introducing the following notation: For $z= re^{i \theta} \in \C \backslash \{0\}$ and $s \geq 0$ define \begin{equation}\label{eqn4}
z^{\langle s\rangle}  :=  r^s e^{-i \theta}.
\end{equation}
If $z=0$ then $z^{\langle s \rangle}$ is defined to be $0$. A way to interpret this is that the mapping $z \mapsto  z^{\langle s \rangle}$ generalizes complex conjugation.
With this notation, \eqref{eqn3} tells us that
 \begin{equation}\label{eqn100}
 f \perp_p g \iff \sum_{k  = 0}^{\infty} a_k ^{ \langle p-1 \rangle} b_k =0.\end{equation}
 
In connection with this notation for powers, we will need the following derivative formulas in the sections ahead.  Their verification entails routine calculus.
 \begin{Lemma}\label{derivform}
 If $t$ is a real variable, and $s > 1$, then
 \[
    \frac{d}{dt} |t|^s = st^{\langle s-1 \rangle}\ \ \mbox{and}\,\ \
    \frac{d}{dt} t^{\langle s \rangle} = s|t|^{s-1}.
 \]
 
 \end{Lemma}
 
 \section{Basic properties of $\Omega_p$}
 
Fix $f \in \ell^{p}_{A}$ and consider the finite dimensional vector space $f \mathscr{P}_n$, where we recall that $\mathscr{P}_n$ are the polynomials of degree at most $n$.
 Since $f \mathscr{P}_n $ is closed (being finite dimensional), \eqref{0iuyghjkl} says there is a unique $p_{n, f} \in \mathscr{P}_n$ such that 
$$\|1 - p_{n, f} f\|_p = \operatorname{dist}(1, f \mathscr{P}_n),$$ and that 
$p_{n, f} f$ is the metric projection of $1$ onto the subspace $f \mathscr{P}_n$.

\begin{Definition}\label{originalpolynh}
For $f \in \ell^{p}_{A}$ and $n \in \mathbb{N}_0$, the unique $p_{n, f} \in \mathscr{P}_n$ for which $$\|1 - p_{n, f} f\|_p = \operatorname{dist}(1, f \mathscr{P}_n)$$
is called the {\em optimal polynomial approximant (OPA) of $1/f$ of degree at most $n$}, abbreviated $p_{n, f}$. We call $p_{1, f}$ the {\em optimal linear approximant.}
\end{Definition}

\begin{Remark}\label{Rnozerosisi2} Here are some simple observations about $p_{n,f}$.
\begin{enumerate}
\item If $f$ is a constant function, then $p_{n, f}$ is a constant polynomial for all $n \in \N_{0}$. 
\item If $\operatorname{deg} p_{n, f}  = k < n$, then $p_{n, f} = p_{t, f}$ for all $k \leq t \leq n$. 
\end{enumerate}
\end{Remark}

The following helps us understand why the assumption $f(0) \neq 0$ is made in numerous results to follow, in order to eliminates trivial cases.

\begin{Proposition} \label{Rnozerosisi2}
For $f \in \ell^{p}_{A}$, the following are equivalent:
\begin{enumerate}
\item $f(0)=0$.
\item $p_{n,f} \equiv 0$  for all $n \in \N_{0}$.
\item There exists some $n \in \N_0$ such that $p_{n,f} \equiv 0$ .
\end{enumerate}
\end{Proposition}
\begin{proof} Clearly $(b) \Rightarrow (c)$. To check $(a) \Rightarrow (b)$, take $f(z) = \sum_{k = 0}^{\infty} a_{k} z^k$, and suppose that $a_0 = f(0)=0$. For any $n \in \N_0$ and $\phi \in \mathscr{P}_n$,
let $(c_k)_{k = 0}^{\infty}$ be the sequence of Taylor coefficients of $\phi  f$. Then
$$\|1 - \phi  f\|_{p} \geq \Big(1 + \sum_{k = 1}^{\infty} |c_k|^p \Big)^{\frac{1}{p}} \geq 1 = \|1 - 0 f\|_{p}.$$
Thus, $p_{n, f} \equiv 0$ for all $n \in \N_0$. 
To see that $(c) \Rightarrow (a)$, suppose to the contrary that $f(0) \neq 0$. Then, for any $c \in \C$, 
$$\|1 - c f\|_{p}^{p}  = |1 - c f(0)|^p + |c|^p \|f - f(0)\|_{p}^{p}.$$
For  $c$ sufficiently close to zero, the above expression can be made less than $1$. Thus $p_{0, f} \not \equiv 0$.
Since 
$$\inf_{\phi \in \mathscr{P}_n} \|1 - \phi  f\|_{p} \leq \|1 - p_{0, f} f\|_{p} < 1,$$
it follows from (b) that $p_{n, f} \not \equiv 0$ for all $n$. 
\end{proof}

Next, we see that OPAs are well behaved with respect to removing a root.
\begin{Lemma}\label{easy1}
Let $f \in \ell^{p}_{A}$ and $n \in \N$. Suppose $z_0 $ is a root of $p_{n, f}$. Then 
$g(z)=f(z)\, (z-z_0)$ belongs to $\ell^p_A$ and
$$p_{n - 1,g}(z) = \frac{p_{n, f}(z)}{z-z_0}.$$
\end{Lemma}

\begin{proof}
From Definition \ref{originalpolynh}, 
\begin{align*}
\inf_{q \in \mathscr{P}_{n - 1}} \|1 - q g\|_{p}
& = \inf_{q \in \mathscr{P}_{n - 1}} \|1 -  q(z) (z - z_0) f(z)\|_{p}\\ & \geq \inf_{\phi \in \mathscr{P}_{n}} \|1 - \phi(z) f(z)\|_{p}\\
& = \|1 - p_{n, f}(z) f(z)\|_{p}.
\end{align*}
Now observe that 
$$q(z) = \frac{p_{n, f}(z)}{z - z_0}$$
minimizes the left-hand side. By Definition \ref{originalpolynh}, $q = p_{n - 1,g}$. 
\end{proof}
Repeated application of Lemma \ref{easy1} yields the statement below.
\begin{Corollary}
For $z_0 \in \C$, the following are equivalent. 
\begin{enumerate}
\item There is an $f \in \ell^{p}_{A}$ with $f(0) \not  = 0$ and an $n \in \N_0$ such that $p_{n, f}(z_0) = 0$. 
\item There is an $f \in \ell^{p}_{A}$ with $f(0) \not  = 0$ such that $p_{1, f}(z_0) = 0$. 
\end{enumerate}
\end{Corollary}

The previous corollary is a significant reduction of the problem of describing $\Omega_p$ since one just needs to focus on finding the  zeros of optimal {\em linear} approximants. For a given $1<p<\infty$, we will see that not every complex number $z_0$ arises as the root of an OPA.  With this in mind we can redefine the region of the plane containing the eligible roots.
\begin{Corollary}
For $1 < p < \infty$, 
$$\Omega_{p}  :=  \{z_0 \in \C: \exists f \in \ell^{p}_{A}, f(0) \not = 0 \ni p_{1, f}(z_0) = 0\}.$$
\end{Corollary}


Let us establish some basic geometric properties of $\Omega_p$. We start by connecting $p_{1,f}$ to a simpler minimization problem.  Here, we write $p_{1,f}(z) = c(1 - t_f z)$, so that $p_{1,f}$ has a root at $z= 1/t_f$.

\begin{Lemma}\label{87IUYUYUYUYUbnbnbn}
If $f \in \ell^{p}_A$ with $f(0) \neq 0$, then 
  $(1 - t_f z)f(z) \perp_{p}  z f(z)$, and
  $$\|(1 - t_f z) f(z)\|_{p} = \min_{t \in \C} \|(1 - t z) f(z)\|_{p}.$$
\end{Lemma}

\begin{proof}
Since 
$$\|1 - p_{1, f}(z) f(z)\|_{p} \leq \|1 - p_{1, f}(z) f(z) + \alpha f(z)\|_{p}$$
and
$$ \|1 - p_{1, f}(z) f(z)\|_{p} \leq \|1 - p_{1, f}(z) f(z) + \alpha z f(z)\|_{p}$$
hold for all $\alpha \in \C$, 
the definition of Birkhoff-James orthogonality (Definition \ref{BJ_defn}) implies that 
\begin{align}
1 - p_{1, f}(z) f(z) &\perp_{p} f(z) \label{firstperpcondz}\\
1 - p_{1, f} f(z) &\perp_{p} z f(z).\label{secperpcondz}
\end{align}
Next, \eqref{secperpcondz} can be written as 
$$
1 - c(1-t_fz) f(z) \perp_{p} z f(z).
$$
Since the constant term in $z f(z)$ is zero, one can use Lemma \ref{BJsemi} to see that 
$c (1 - t_f z)f(z) \perp_{p} z f(z)$. Another application of Lemma \ref{BJsemi} says we can also drop the $c$ on the left side to obtain $(1 - t_f z)f(z) \perp_{p}  z f(z)$.   The second claim follows immediately from the first.
\end{proof}

Here is a simple example of a non-inhabitant of $\Omega_p$. 

\begin{Proposition}\label{punct}
$0 \not \in \Omega_p$ for all $1 < p < \infty$.
\end{Proposition}

\begin{proof}
With the hypothesis that $f(0) \not = 0$, it must be that $p_{1, f}(0) \not = 0$. Otherwise, $p_{1, f}(z) = c z$ and so 
$$\|1 - p_{1, f} f\|^{p}_{p} = 1 + |c|^{p} \|f\|_{p}^{p}.$$
The right hand side of the above is minimized when $c = 0$, which would make $p_{1, f} \equiv 0$. But we already excluded this possibility.
\end{proof}

Here are some inhabitants of $\Omega_p$. 

\begin{Proposition}\label{76y78uijsdfg}
$\C \backslash\overline{\D} \subset \Omega_p$ for all $1 < p < \infty$.
\end{Proposition}

\begin{proof}
For $	z_0 \in \C \backslash\overline{\D}$,
set  $f(z)= (z - z_0)^{-1}$ which, under the hypothesis that $|z_0| > 1$, belongs to $\ell^p_A$. Then $p_{1, f}(z)=z-z_0$, since $1-p_{1, f}f \equiv 0$.
\end{proof}

Therefore, the problem of describing $\Omega_p$ reduces to looking at $\Omega_{p} \cap \overline{\D}$. When $p = 2$, there is the following special case of a result from  \cite{BenJLMS}. We include a proof here for completeness.

\begin{Proposition}\label{prop311}
$\Omega_2 = \C \backslash\overline{\D}$.
\end{Proposition}

\begin{proof}
When $p = 2$, Birkhoff-James orthogonality $\perp_{2}$ agrees with Hilbert space orthogonality. From Proposition \ref{76y78uijsdfg} we just need to show that $ \Omega_2 \cap \overline{\D} = \varnothing$.  From Lemma \ref{87IUYUYUYUYUbnbnbn} for $p=2$ we have
\begin{align*}
    f(z) &= (1 - t_f z)f(z)  +  t_f z f(z)\\
    \langle zf(z), f(z) \rangle &= \langle zf(z), (1 - t_f z)f(z) \rangle + t_f \langle zf(z), zf(z) \rangle \\
     \langle zf(z), f(z) \rangle &= 0 + t_f \|f\|_2^2. \\
\end{align*}
Consequently $t_f$ satisfies 
$$
|t_f| = \frac{|\langle z f(z), f(z) \rangle|}{\|f\|_{2}^{2}} \leq \frac{\|z f(z)\|_{2} \|f\|_{2}}{\|f\|_{2}^{2}} = \frac{\| f\|_{2} \|f\|_{2}}{\|f\|_{2}^{2}} = 1.$$
 Equality is  attained precisely when $f(z)$ and $z f(z)$ are linearly dependent, and this is exactly when $f \equiv 0$. 
\end{proof}

As we proceed through the paper, we will show that when $p \not = 2$, $\Omega_p$ is the complement of a disk of radius less than one and thus, unlike the $p = 2$ case, $\Omega_{p} \cap \D \not = \varnothing$.

Next we show that $\Omega_p$ has the following rotational symmetry property. 

\begin{Proposition}\label{1hHyyg667trRR}
 $e^{i \theta} \Omega_p = \Omega_p$ for all $0 \leq \theta < 2 \pi$.
\end{Proposition}

\begin{proof}
 If $f \in \ell^{p}_A$, $q \in \mathscr{P}_{n}$, and $\theta \in [0, 2 \pi)$, the nonzero Taylor coefficients $(c_k)_{k = 1}^{\infty}$ of $1-q(z)f(z)$ and those of $1-q(e^{i\theta} z)f(e^{i\theta}z)$ have the same modulus. Therefore, 
 \begin{align*}
 \|1  - q(z) f(z)\|_{p}^{p} & = |1 - q(0) f(0)|^p + \sum_{k = 1}^{\infty} |c_k|^p\\
 & = \|1 - q(e^{i \theta} z) f(e^{i \theta} z)\|_{p}^{p}.
 \end{align*}
Thus,   $p_{n, f}(z)$ is optimal for $f(z)$ if and only if $p_{n, f}(e^{i\theta} z)$ is optimal for $f(e^{i\theta}z)$. Finally, $z_0$ is a zero of $p_{n, f}(z)$ if and only if $e^{-i\theta} z_0$ is a zero of $p_{n, f}(e^{i \theta} z)$. 
\end{proof}

Our last goal in this section is to show that $\Omega_p$ is path connected. Together with Propositions \ref{punct} and \ref{76y78uijsdfg}, and the rotational symmetry of $\Omega_p$, this will show that $\Omega_p$ is either the punctured plane or the complement of some disk (either open or closed) centered at the origin. The path connectedness of $\Omega_p$ will be a consequence of the following lemma.

\begin{Lemma}  \label{continuityofrootlemma}
Let $[a,b] \subset \mathbb{R}$.  
Suppose that $F: \mathbb{R}\times\mathbb{R}\longmapsto \mathbb{R}$ is continuous and satisfies the property that
for each $x \in \mathbb{R}$, there exists a unique $t_x \in [a,b]$ such that $F(x, t_x) = 0$.  Then the mapping $\Upsilon:  x \mapsto t_x$ is continuous.
\end{Lemma}
\begin{proof}
Observe that $\Upsilon$ is well defined.  Suppose that $x_n \to x$ in $\mathbb{R}$.  The continuity of $F$ shows that the set
\[
      S := \{(x,t):\ F(x,t) = 0\} = F^{-1}(\{0\})
\]
is closed.   

Consider the collection of pairs $\{ (x_n, \Upsilon(x_n)) \}$.  The collection of second entries lies in the bounded interval $[a,b]$, and hence there is an accumulation point $y \in [a,b]$, and a subsequence $(x_{n_k})_{k \geq 1}$ for which $\Upsilon(x_{n_k}) \to y$.   The closedness of $S$ then implies that $\Upsilon(x) = y$.   In a similar way we see that any subsequence of $(x_n)_{n \geq 1}$ has a further subsequence $(x_{\nu_k})_{k \geq 1}$ for which 
$
       \Upsilon(x_{\nu_k}) \to \Upsilon(x).
$
This implies that 
$
        \Upsilon(x_n) \to \Upsilon(x).
$
\end{proof}

\begin{Corollary}\label{condforextzer2}
$\Omega_{p}$ is path connected for any $1 < p < \infty$.
\end{Corollary}

\begin{proof}
From the rotational symmetry of $\Omega_p$, it suffices to show that $\Omega_p \cap \R^+$ is an interval. To apply the previous lemma,  fix a polynomial $f(z) = \sum_{k = 0}^{d} a_k z^k$ such that its linear approximant has a real positive zero, and define 
$$F(x, t) = \frac{d}{dt}\Big\|(1 - t z) \big(x + \sum_{k = 1}^{d} a_k z^k\big)\Big\|_{p}^{p}=: h'(t), \ \,(x,t) \in \mathbb{R}\times\mathbb{R}.$$
Then $F$ satisfies the hypothesis of the lemma and for each $x$ the corresponding $t_x$ is the zero of $h'$. The intermediate value theorem says that these values of $t_x$ constitute an interval.  Repeat this argument, changing the definition of $F$ by replacing $a_j$ by $x$.  We see that the solution $t$ of $h'(t) = 0$ depends continuously on the each of coefficients $a_j$ of $f$. 
\end{proof}

\section{$\Omega_p$ and $p$-inner functions}\label{Sect3}

A point $z_0$ belongs to $\Omega_p$ precisely when there is an $f \in \ell^{p}_{A}$ with $f(0) \not = 0$ such that $p_{1, f}(z_0) = 0$.  As we have seen earlier, and will reinforce later in this section, the problem of finding the zero of $p_{1,f}$ reduces to the extremal problem of
finding a $t_{f} \in \mathbb{C}$ such that
$$\|(1 - t_{f}z)f(z)\|_{p} = \min_{t \in \C} \|(1 - t z) f(z)\|_{p}.$$ Furthermore, the definition of $\Omega_{p}$ its rotational symmetry imply that $e^{i \theta}/t_f \in \Omega_p$ for every $\theta \in [0, 2 \pi)$. We now connect this with another  extremal problem involving the concept of a $p$-inner function. 
 Most of the material in this section comes, {\em mutatis mutandis}, from the book \cite[Ch.~8]{CMR}.
 
 \begin{Definition}\label{p_INNER}
 A function $f \in \ell^{p}_{A} \backslash\{0\}$ is {\em $p$-inner} if  
 $$f \perp_{p} S^{n} f  \quad \mbox{for all $n \geq 1$}.$$
\end{Definition}

Examples of $p$-inner functions include the monomials $\{z^n: n \in \N_0\}$ as well as 
$$f_{w}(z) = \frac{1 - z/w}{1 - w^{\langle p' - 1\rangle} z}, \quad \mbox{where $w \in \D \backslash\{0\}$}.$$
In the above,  $p'$ denotes the H\"{o}lder conjugate index to $p$ \cite[p.~111]{CMR} and $w^{\langle p' - 1\rangle} $ comes from \eqref{eqn4}.

When $p = 2$, $\ell^{2}_{A}$ is a Hilbert space with inner product 
$$\langle f, g\rangle = \sum_{k = 0}^{\infty} a_{k} \overline{b_k},$$
where $f(z) = \sum_{k=0}^{\infty}a_kz^k$ and $g(z) = \sum_{k=0}^{\infty}b_k z^k$.
By Parseval's theorem, this can be written in integral form as 
$$\langle f, g\rangle = \int_{0}^{2 \pi} f(e^{i \theta}) \overline{g(e^{i \theta})}\,\frac{d \theta}{2 \pi}.$$
The criterion for $2$-inner from Definition \ref{p_INNER} becomes 
$$\int_{0}^{2 \pi} |f(e^{i \theta})|^2 e^{i n \theta} \frac{d \theta}{2 \pi} = 0 \quad \mbox{for all $n \geq 1$}.$$
The equation above, along with its complex conjugate (and with a consideration of the Fourier coefficients of $|f(e^{i \theta})|^2$), shows that a function $f$ is $2$-inner precisely when $|f(e^{i \theta})|$ is a nonzero constant almost everywhere; that is to say, apart from a nonzero multiplicative constant, $f$ is inner in the traditional sense from Beurling's  paper \cite{MR27954}.

For $f \in \ell^p_A\backslash\{0\}$ let
$$[f]  :=  \overline{\operatorname{span}}\{S^n f: n \in \N_0\}$$ denote the shift-invariant subspace of $\ell^p_A$ generated by $f$.  Let 
 $$J = f - \widehat{f},$$ where 
 $\widehat{f}$ is the unique function in $\ell^{p}_{A}$ for which 
$$\|f - \widehat{f}\|_{p} = \inf_{g \in [S f]} \|f - g\|_{p}.$$  It turns out that $J$
is $p$-inner
and  \cite[p.~106]{CMR} shows that every $p$-inner function arises in this manner. 

Classical theory shows that inner functions (in the sense of Beurling) satisfy certain extremal problems. It turns out that $p$-inner functions do something analogous \cite[p.~103]{CMR}. 

\begin{Proposition}\label{iiiMMMi}
For $f \in \ell^{p}_{A}$ with $f(0)  = 1$, let 
\begin{equation}\label{extremeeq1}
M  :=  \sup\left\{ |\phi(0)| : \|\phi\|_p =1, \phi \in [f] \right\}
\end{equation}
and 
\begin{equation}\label{extremeeq2}
I  :=  \inf \left\{ \|g\|_p : g(0)=1, \ g \in [f] \right\}.
\end{equation}
Then,
\begin{enumerate}
\item $I M = 1$;
\item Each of the extremal problems above  have unique solutions $\phi$ and $g$ respectively with 
$$\phi(z) = \frac{g(z)}{\|g\|_{p}} \quad \mbox{and} \quad g(z) = \frac{\phi(z)}{\phi(0)}.$$
\end{enumerate}
\end{Proposition}

Let $G$ be the (unique) solution to the infimum problem (\ref{extremeeq2}). Then by definition,
\[
\|G\|_p \leq \| G(z) + z\Psi(z) \|_p \quad \mbox{for all 
$\Psi \in [f]$.}
\]
 Conversely, this condition characterizes $G$. 
In particular, we have $G \perp_p S^n G$ for every $n\in\mathbb{N}$, and thus $G$ is $p$-inner.
In fact $G = f - \widehat{f}$.

The above propositions also hold, with essentially the same proof, if instead we consider the extremal problems
\begin{equation}\label{extremeeq11}
M_{N}  :=  \sup\left\{ |\phi(0)| : \|\phi\|_p =1, \phi \in f \mathscr{P}_N \right\}
\end{equation}
and
\begin{equation}\label{extremeeq21}
I_{N}  :=  \inf \left\{ \|g\|_p : g(0)=1, g \in f \mathscr{P}_N \right\}
\end{equation}
for fixed $N \in \N$ and $f \in \ell^{p}_{A}$.

\begin{Proposition}
Suppose that $f \in \ell^{p}_{A}$, $f(0) = 1$ and $N \in \mathbb{N}$.   Then
\begin{enumerate}
\item $I_{N} M_{N} = 1$.
\item Each of the extremal problems \eqref{extremeeq11} and \eqref{extremeeq21} have unique solutions $\phi_N$ and $J_N$ respectively which differ by a nonzero multiplicative constant. 
\end{enumerate}
\end{Proposition}

We retain the notation $J_N$ for the unique solution to the infimum problem (\ref{extremeeq21}) since it ties into our typical usage of $J = f - \widehat{f}$ for the $p$-inner function associated with $f$.
Then, by definition,
\[
\|J_N\|_p \leq \| J_N(z) + z\Psi(z) \|_p  \quad \mbox{for all 
$\Psi \in f \mathscr{P}_{N - 1}$.}
\]
Conversely, this condition characterizes $J_N$. 

For fixed $N \in \N$, recall the original optimal polynomial approximation problem from Definition \ref{originalpolynh}:
\begin{equation}\label{bbbbcbb1b1b11}
\inf\{\|1 + Pf \|_{p}: P \in \mathscr{P}_N \}.
\end{equation}

By the nearest point property for uniformly convex spaces, there is a unique $\Phi_N \in \ell^{p}_{A}$ for which $\Phi_N -1 \in f \mathscr{P}_N$ and the infimum in \eqref{bbbbcbb1b1b11} is attained: 
\begin{equation}\label{ooo89we754}
\|\Phi_N\|_{p} = \inf\{\|1 + Pf \|_{p}: P \in \mathscr{P}_N \}.
\end{equation}
An argument similar to the one in \cite[Ch.~8]{CMR} shows that 
$$\Phi_N = 1 - \frac{J_N}{1 + (\|J_N\|_{p}^{p} - 1)^{p' - 1}}.$$

 Write $\Phi_N = 1 - P_N f$, thereby defining the OPA $P_N = p_{N, f}$. The preceding argument justifies the following result:
\begin{Theorem}\label{invsubs} Let $N \in \N$ and $1<p<\infty$. Then
\[
p_{N, f} =  \frac{J_N/f}{1 + (\|J_N\|_{p}^{p} - 1)^{p' - 1}}.
\]
\end{Theorem}

This extends Theorem 2.4 in \cite{RemBen} to the context of $\ell^p_A$ spaces.
Notice that $J_N/f$ is a polynomial and $p_{n, f}$ has the same roots (counting multiplicities, of course). Any extra zeros (i.e., those not of $f$) of $J_N$ will also be extra zeros of $p_{N, f}$. 

Of special importance here is the case where $N = 1$. Here, for $f \in \ell^{p}_{A}$ with $f(0) = 1$, we have 
$$I_{1} = \inf\{\|g\|_{p}: g(0) = 1, g \in f \mathscr{P}_1\} = \inf\{\|(1 - t z) f(z)\|_{p}: t \in \C\}.$$
From the definition of $J_1$ as the unique solution to the problem, we see that 
$J_{1}(z) = (1 - t_f z) f(z).$


\section{Regions excluded from $\Omega_p$}

At this point, the skeptical reader might wonder about the possibility that $\Omega_p = \C \backslash\{0\}$. The purpose of this section is to determine radii $r_p \in (\tfrac{1}{2},1)$ such that $\Omega_p \cap r_p \D = \varnothing$. This will be accomplished by using Birkhoff-James orthogonality in $\ell^{p}_{A}$ along with the  Pythagorean inequalities.

Two important tools in our analysis are the forward and backward shift operators defined on $\ell^{p}_{A}$ by 
\begin{equation}\label{987uyhujhnjui}
(S f)(z)  :=  z f(z) \quad \mbox{and} \quad (B f)(z)  :=  \frac{f(z) - f(0)}{z}.
\end{equation}
These operators  are easily checked to be bounded on $\ell^{p}_{A}$ with 
$\|B\|  = \|S\| = 1$. Furthermore, $S$ is an isometry on $\ell^{p}_{A}$. 
Related to these operators is the difference quotient operator \cite[p.~91]{CMR}.

\begin{Lemma}\label{876thjnbhyy11oOO}
If $w \in \D$ the difference quotient operator 
$$(Q_{w} f)(z)  :=  \frac{f(z) - f(w)}{z - w}$$ is bounded on $\ell^{p}_{A}$ with norm $\|Q_{w}\|$ which satisfies 
$$\|Q_{w}\| \leq \frac{1}{1 - |w|}.$$
\end{Lemma}

Another important tool for us is the following extension of the Pythagorean theorem to $\ell^{p}_{A}$ spaces, which takes the form of a family of inequalities \cite[p.~58]{CMR}.

\begin{Lemma}\label{Pytha}
\hfill
\begin{enumerate}
\item For $2 \leq p < \infty$,
$$ f \perp_p g \implies \|f\|_{p}^{p}  + \frac{1}{2^{p - 1} - 1} \|g\|_{p}^{p} \leq \|f +g\|_{p}^{p}.$$
\item For $1 < p \leq 2$, 
$$ f \perp_p g \implies \|f\|_{p}^{2}  +(p - 1) \|g\|_{p}^{2} \leq \|f + g\|_{p}^{2}.$$
\end{enumerate}
\end{Lemma}

The following proposition says that certain disks centered at the origin are excluded from $\Omega_p$.

\begin{Proposition}\label{emptydisks}
\hfill
\begin{enumerate}
\item For $p \geq 2$, 
$\Omega_p \cap \frac{1}{s} \overline{\D} = \varnothing,$ whenever 
\begin{equation}\label{exsloue88y}
(s - 1)^p + \frac{s^p}{2^{p-1}-1} \geq 1.
 \end{equation}
\item For $1 <p< 2$, 
$\Omega_p \cap \frac{1}{s} \overline{\D}  = \varnothing,$
whenever \begin{equation}\label{exsloue88y22}
s \geq (2/p)^{1/p}.
\end{equation}
\end{enumerate}
\end{Proposition}

\begin{proof}
Firstly, assume that $p \geq 2$. By Lemma \ref{87IUYUYUYUYUbnbnbn} write 
$f(z) = (1 - t z)  f(z) + t z f(z)$ and note that the two summands are Birkhoff-James orthogonal.

Consequently, by Lemma \ref{Pytha},
\begin{equation}\label{09876ytuik}
       \|f\|^p_p  \geq  \|(1-tz)f(z)\|^p_p  +  \frac{1}{2^{p-1}-1} \|tzf(z)\|^p_p,
\end{equation}
Next, 
\begin{align}
     f(z)  &= \frac{ (1-tz)f(z)}{1 - tz} \nonumber\\
       &=  \frac{ (1-tz)f(z) -  (1 - t/t)f(1/t)}{(-t)(z - 1/t)} \nonumber \\
       &=  (-1/t)Q_{1/t}([1 - tz]f(z)),\label{bvghnmjnh7uuHYY77}
\end{align}
where $Q_{1/t}$ is the difference quotient operator at the point $1/t$ from Lemma \ref{876thjnbhyy11oOO}. Notice we needed $1/t \in \mathbb{D}$. Therefore,
 \begin{align*}
  \|f\|^p_p  & \geq  \|(1-tz)f(z)\|^p_p  +  \frac{1}{2^{p-1}-1}\|tzf(z)\|^p_p && \mbox{(by \eqref{09876ytuik})}\\
   &   \geq  \|(1-tz)f(z)\|^p_p  +  \frac{|t|^p}{2^{p-1}-1} \|f(z)\|^p_p  &&  \mbox{($S$ is an isometry)}
   \end{align*}
   This implies that 
      $$ \Big(1 -  \frac{|t|^p}{2^{p-1}-1}\cdot 1  \Big)\|f\|^p_p \geq \|(1-tz)f(z)\|^p_p.$$
      Now use \eqref{bvghnmjnh7uuHYY77} to get 
     $$\Big(1 -  \frac{|t|^p}{2^{p-1}-1}  \Big)\|(-1/t)Q_{1/t}([1 - tz]f(z))\|^p_p \geq \|(1-tz)f(z)\|^p_p$$
     from which 
      $$\Big(1 -  \frac{|t|^p}{2^{p-1}-1}  \Big)(1/|t|^p)\|Q_{1/t}\|^p \geq 1.$$
      Lemma \ref{876thjnbhyy11oOO} yields
      $$\Big(1 -  \frac{|t|^p}{2^{p-1}-1}  \Big)\frac{(1/|t|^p)}{(1 - |1/t|)^p} \geq 1.$$
This yields the bound in \eqref{exsloue88y},
which proves (a).

 A similar calculation, but using the second Pythagorean inequality in Lemma \ref{Pytha}, yields (b).
\end{proof}

Since both inequalities  in the previous proposition are met when $s=2$, we have the following:
\begin{Corollary}\label{nozerosonehalfff}
$\Omega_{p} \cap \frac{1}{2} \overline{\D} = \varnothing$ for any $1 < p < \infty$.
\end{Corollary}

Below is a table of the values of $s$ satisfying \eqref{exsloue88y} and \eqref{exsloue88y22} and the corresponding exclusion regions $r\D = \{|z| < 1/s\}$ with $r\D\cap \Omega_p = \varnothing$.

$$\begin{tabular}{|c|c|c|}
       $p$ & $s \geq$ & $r$\\ \hline
          1.50 &  1.21141 & 0.825482\\
           1.66 & 1.11560 & 0.896378\\
                  1.75 & 1.07929 & 0.926535\\
                     1.80 & 1.06028 & 0.943147\\
       1.83 & 1.04861 & 0.953648\\
       2.1 & 1.06436 & 0.939533\\
       4 & 1.57890 & 0.633368\\
       6 & 1.72617 & 0.579318\\
       8 & 1.79348 & 0.557577\\
       10 &  1.83319  & 0.545498\\
       12 & 1.85983 & 0.537682\\
       14 & 1.87908 & 0.532175\\
       16 & 1.89367 & 0.528076
\end{tabular}$$
Notice how the value of $s \to 1$ as $p \to 2$, consistent with Proposition \ref{prop311}, which says that $\Omega_{2} \cap \overline{\D} = \varnothing$. 

The values of $s$ in the above table and in Proposition \ref{emptydisks} are not optimal, however, in the sense that they determine $\Omega_p$ completely.  Instead they furnish a simple bound for the extent of $\Omega_p$.  Much more work remains to be done; in the next section  we show that $\Omega_{p} \cap \D \not = \varnothing$, and in Theorem \ref{lastth}  we show that $\Omega_p = \C \backslash \frac{1}{\tau_{p} }\overline{\D}$.  The constant $\tau_{p}$ will be given implicitly. There is no closed formula for it and it is challenging to compute it numerically.

\section{Zeros inside the disk}\label{Sect4}

Proposition \ref{prop311} says that $\Omega_{2} \cap \overline{\D} = \varnothing$. Through some specific examples, this section will show that $\Omega_p \cap \D \neq \varnothing$ when $p \not = 2$. 

\begin{Proposition}\label{propo102}
Let $ 1 < p <2$. If 
 $$f_k(z)= \sum_{j=0}^k (j+1)z^j,$$
then for large enough $k$, the zero of $p_{1, f_k}$ lies in $\D$.
\end{Proposition}

\begin{proof}
From Lemma \ref{87IUYUYUYUYUbnbnbn}, we have that \[1-p_{1, f_k} f_k(z) \perp_{p} zf_k(z).\] Since $p_{1, f}(z)= c (1-t_f z)$ unique, it must have real coefficients (otherwise, since the coefficients of $f_{k}$ are real, $\overline{c}(1-\overline{t_f}z)$ would give another optimal approximant of the same degree).  It follows from Lemma \ref{BJsemi} that a sufficient condition for $t=t_f$  is for $t_f$ to be a zero of
\begin{equation}\label{eqn110} g(t):= \sum_{j=1}^k ((j+1)-tj)^{\langle p-1 \rangle} j + (-t(k+1))^{\langle p-1 \rangle} (k+1).\end{equation}
Then we have 
$$g(2) = - \sum_{j=1}^k (j-1)^{\langle p-1 \rangle} j -2(k+1)^p < 0.$$
It remains to check that $g(1) >0$. Notice that
\[g(1) = \sum_{j=1}^k j - (k+1)^p.\]
Now use the fact that $1 < p  < 2$ to see that \[\sum_{j=1}^k j = \frac{k(k-1)}{2} \gg (k+1)^p\] for large enough $k$. Thus,  for large enough $k$, $g(1) > 0$.
\end{proof}

Let us now deal with the remaining case $p>2$. 

\begin{Proposition}\label{propo103}
Let $p>2$. If  \[f_k(z) = 1 + \sum_{j=1}^{2k} \left(2-\frac{j-1}{k}\right) z^j,\]
then for large enough $k$, the zero of $p_{1, f_k}$ lies in $\D$.
\end{Proposition}

\begin{proof}
Once more, denote by $p_{1,f}(z)=c(1-t_fz)$ the optimal linear approximant for $f_k$. As in the proof of the previous proposition, $c, t_f\in \R$ and \[1-p_{1, f} f_k(z) \perp_{p} zf_k(z).\]

From Lemma \ref{87IUYUYUYUYUbnbnbn}, this translates into the condition on $t_f$ that 
\[h(s) : = \|(1-sz)f_k(z)\|^p_p\]
must be minimized at $s=t_f$.
Since $h$ is a differentiable function of $s$ for all $p>2$, we must have $h'(t_f)=0$ and this condition will determine the value of $t_f$. By convexity, we must have that the limit as $s\rightarrow \infty$ of $h'(s)$ is positive, and thus, for $t_f$ to be larger than $1$ it is enough to show that \begin{equation}\label{eqn300}
h'(1)<0.\end{equation}

Let $a_j$ be  the Taylor coefficient of order $j$ of $f_k$.  Then
\begin{align}
 h'(1) & = \frac{d}{ds} \Big\{ |a_0|^p + \sum_{j=0}^{2k} |a_{j+1}-sa_{j}|^p \Big\}\Big|_{s = 1} \nonumber \\
 &= p \sum_{j=0}^{2k} a_j \left( a_j - a_{j+1}\right)^{p-1}.\label{eqn301}
\end{align}
With the understanding that $a_{2k +1}=0$, we have
$$(a_j -a_{j+1})^{p-1} = 
\begin{cases}
 -1 & \mbox{if } j=0\\
 k^{1-p}&  \mbox{if } 1 \leq j \leq 2 k
 \end{cases}.$$

At the same time $|a_j| \in [0,2]$ for all $j \in \N$. This means that the right-hand side of \eqref{eqn301} is bounded above  by 
$p(-1 +C k^{2-p}),$ where $C>0$ is independent of $k$. For large enough $k$ this is negative and thus the condition in \eqref{eqn300} is met.
\end{proof}

\begin{Corollary}
$\Omega_{p} \cap \D \not = \varnothing$ for $1 < p < \infty$ and $p \not = 2$.
\end{Corollary}

In other words, there are extra zeros inside the disk $\mathbb{D}$ when $p \neq 2$.

\section{A method for finding the optimal radius}

So far we have shown that for any $1 < p < \infty$, $\Omega_p$ is the complement of a certain disk centered at the origin. Moreover, when $p \not = 2$, this disk as radius less than $1$. Our objective in this section is to develop tools for finding the radius of this disk. The rotational symmetry of $\Omega_p$ enables us to focus on the zeros of optimal linear approximants on $\R$, and in fact, on $[\frac{1}{2}, 1]$.  

\subsection{A smaller problem}
A nice reduction of the problem is that we can restrict to studying the case that $f$ is a polynomial, since polynomials are dense in $\ell^p_A$ (from the definition of $\ell^{p}_{A}$) and linear approximants are well behaved under limits. Recall that for $f \in \ell^{p}_{A}$, $t_f$ is the unique complex number that satisfies 
$$\|(1 - t_{f} z) f(z)\|_{p} = \min_{t \in \C} \|(1 - t z) f(z)\|_{p}.$$

\begin{Proposition}
If $f_n \to f \in \ell^{p}_{A}$, then $t_{f_n} \to t_{f}$. 
\end{Proposition}

\begin{proof}
Since $(f_n)_{n = 1}^{\infty}$ is a convergent sequence,
$$M = \sup_{n \geq 1} \|f_n\|_{p} < \infty.$$
By Proposition \ref{emptydisks}, 
$$\sup_{n \geq 1} |t_{f_n}| =  K < \infty.$$  By the latter bound, there is a subsequence $(t_{f_{n_k}})_{k = 1}^{\infty}$ which converges to some $t$. Then for any $0 < |w| < 1$,
\begin{align*}
\|(1 - tz)f(z)\|_p &\leq \|(1 - t_{f_{n_k}} z)f_{n_k}(z)\|_p + |t - t_{f_{n_k}}| \|f_{n_k}\|_p\\
& \qquad  +(1 + |t|) \|f - f_{n_k}\|_p \\
&\leq \|(1 - z/w )f_{n_k}(z)\|_p + |t - t_{f_{n_k}}| M\\
& \qquad +(1 + |t|) \|f - f_{n_k}\|_p \\
&\leq \|(1 - z/w )f(z)\|_p + |t - t_{f_{n_k}}| M\\
& \qquad +(2 + |t|+ 1/|w|) \|f - f_{n_k}\|_p,
\end{align*}
where we have used Young's inequality, and the optimality assumption on $t_{f_{n_k}}$. Now let $k \to \infty$ to see that
\[
\|(1 - tz)f(z)\|_p \leq \|(1 - z/w )f(z)\|_p .
\]
That is, $(1 - tz)f(z)$ is optimal. Such $t$ must be unique, since it arises from estimating a vector by a member of a subspace in a uniformly convex space. Thus, any subsequence of $(t_n)_{n = 1}^{\infty}$ must have a further subsequence that converges to this same $t$. Thus, $t_n \to t$ and $t=t_f$.
\end{proof}

The previous problem allows us to concentrate on the problem of finding, for fixed $1<p<\infty$,
$$ T_{d,p} :=  \sup_{f \in \mathscr{P}_d} |t_f| \quad \mbox{and} \quad \tau_p := \sup_{d \in \N} T_{d,p} .$$
Note that $\tau_{p}$ is what we need in Theorem \ref{MAIN}.

 Is $T_{d,p}$ attained for some $f \in \mathscr{P}_d$?  We argue that it is.  Let $(f_k)_{k = 1}^{\infty}$ be a sequence in $\mathscr{P}_d$ such that 
$t_{f_k} \longrightarrow T_{d,p}$.  Since $t_f$ does not change when $f$ is scaled by a non-zero constant, we may assume that all of the Taylor coefficients of the $f_k$ are uniformly bounded.  Thus, there is a subsequence $(f_{n_k})_{k = 1}^{\infty}$ for which all of the separate coefficient sequences converge.  By relabeling, we may assume that the subsequence is $(f_n)_{n = 1}^{\infty}$.  We have shown that $f_n \longrightarrow f^*$ in $\ell^{p}_A$, for some polynomial $f^* \in \mathscr{P}_d$, and $t_{f^*} = T_{d,p}$.   

This also shows that $1/T_{d,p}  \in \Omega_p$ for all $d$.

To analyze this extremal problem further, we once again consider for fixed $1 < p < \infty$ and $f \in \ell^p_A$ the function 
$$  h(t)  = h_{p,f}(t)  :=   \|(1 - t z) f(z)\|_{p}^{p}.$$

Thus $h$ has a unique minimum at $t_f= 1/z_0$, where $z_0$ is the zero of $p_{1,f}$.

This next result says in our search for the extremal functions $f$ for which maximal values of $|t|$ are attained, we can assume that the coefficients of $f$ are in $\R^+$. 

\begin{Proposition}\label{propo202}
If $p\neq 2$ and $f(z) = \sum_{k=0}^{\infty} a_k z^k \in \ell^p_A$, then $|t_f| \leq |t_g|$, where
\begin{equation}\label{defofg}
      g(z) := \sum_{k=0}^{\infty} |a_k| z^k.
\end{equation}

In particular, for $d \in  \N$ and $d > 2$,  there is an $f^{*} \in \mathscr{P}_d$ whose coefficients are nonnegative and such that $T_{d,p}  = t_{f^{*}}$. 
\end{Proposition}

\begin{proof}  By replacing $f(z)$ with $f(e^{i\gamma}z)$, for a suitable value of $\gamma \in \mathbb{R}$, we may assume that $t_f > 0$.  Let 
\[
    h(t) := \|(1 - tz)f(z)\|_p^p =  |a_0|^p +\sum_{k=1}^{\infty} |a_k - ta_{k-1}|^p,
\]
where $t$ is a real variable.
Then
\begin{align*}
     -h'(t)/p  &=  \sum_{k=1}^{\infty}\big( |a_k|^2-2t\Re(\bar{a}_ka_{k-1}) + t^2|a_{k-1}|^2 \big)^{\frac{p}{2}-1}\big(\Re(\bar{a}_k a_{k-1}) - t|a_{k-1}|^2 \big)\\
      &=  \sum_{k=1}^{\infty} (a_k - ta_{k-1})^{\langle p-1 \rangle}a_{k-1}.
\end{align*}
 To check that convergence is not an issue, we apply H\"{o}lder's inequality with the effect
\begin{align*}
      |h'(t)|  &\leq p \sum_{k=1}^{\infty} |a_k - ta_{k-1}|^{p-1} |a_{k-1}|\\
      &\leq p \Big( \sum_{k=1}^{\infty} |a_k - ta_{k-1}|^{q(p-1)} \Big)^{1/q}
           \Big(  \sum_{k=1}^{\infty}  |a_{k-1}|^p \Big)^{1/p}\\
      &< \infty,
\end{align*}
since $q(p-1) = p$.

The function $\phi(t) = |t|^p$ is convex in $t$ for any $p>1$, and so are its translates and sums, as well as limits of such functions.  
Hence $h(t)$ is convex, and consequently its derivative $h'(t)$ is monotone nondecreasing in $t$.

Setting that aside for the moment, suppose that $H$ is a real function on $\mathbb{R}$ such that $H'$ exists and is monotone nondecreasing, and that the expression 
\[
H(t) + |at - b|^p
\]
attains its minimum value when $t = T$.   Assume further that $a>0$ and $c<b$.

By elementary calculus,
\[
H'(T) + a|aT- b|^{p-2}(aT-b) = 0.
\]

There is a value of $t$, call it $t^*$, such that
\[
H'(t^*) + a|at^* - c|^{p-2}(at^*-c) = 0,
\]
where $b$ has been replaced by $c$.
Since $c < b$ we have 
we have
$$
at -c > at-b$$ and thus 
$$a|at-c|^{p-2}(at-c) > a|at-b|^{p-2}(at-b)$$
(the function $\psi(t) = t^{\langle p-1 \rangle}$ is monotone increasing when $p>1$).

We see that $H'(t) + a|at-c|^{p-2}(at-c) = 0$ for some $t$ satisfying $H'(t) < H'(T)$. By the assumption that $H'(t)$ is nondecreasing, this means that $t^* < T$. In other words, the minimum is achieved at a lower value of $t$.

Now apply the above observation to the function
\[
H(t) = |a_0|^p + \sum_{k=1}^{\infty} (|a_k| - t|a_{k-1}|)^p,
\]
where we are identifying $a = |a_{0}|$, $b = |a_{1}|$, and $c = |a_{1}|\cos\theta$.  Notice that $H(t)$ is just $h_{p,g}(t)$, where $g$ is given by \eqref{defofg}.

The conclusion is that $H(t)$ is critical at a higher value of $t$ than the function
\[
|a_0|^p + (|a_1|\cos\theta - t|a_{0}|)^p +  \sum_{k=2}^{\infty} (|a_k| - t|a_{k-1}|)^p  
\]
for any choice of $\theta$.  Repeat this argument successively for each term in $H(t)$.   That is, replace $|a_k| - t|a_{k-1}|$ with $|a_k|\cos \theta_k - t|a_{k-1}|$ in the $k$th instance.  The conclusion is $H(t)$ is critical at a greater value of $t$ when $\theta_k$ is chosen to be zero.
If we choose the values of $\theta_k$ along each step so that the final function is  $h_{p,f}(t)$, we find that $f$ can be the function that optimizes $t$ only if its coefficients are nonnegative.

In particular, if we restrict our attention to the case $f$ is a polynomial of degree up to $d$, then 
we have already seen that $T_{d,p}$ is attained.  The above argument applies, showing that we may assume that the extremal $f^*$ has nonnegative coefficients.
\end{proof}

\subsection{Lagrange multipliers and recurrence relations}\label{Sect6}

For each $1 < p < \infty$, and each $d \in \N$, we have seen that there is a polynomial $f \in \mathscr{P}_d$ for which the optimal value $T_{d,p} $ is attained.  Furthermore, we reduced the search for $f$ to polynomials with nonnegative coefficients. 
This enables us to use the Lagrange multiplier method to solve for the optimizing polynomial $f$, and compute the value of $T_{d,p} $.   

Given 
$f(z)=\sum_{k=0}^d a_{k} z^k,$
let $t_f=t(a_0,...,a_d)$ be that value of $t$ that minimizes
\[
     \|(1 - tz)f(z)\|_p.
\]  We want to find
\[T_{d,p} := \max_{a_0, a_1, \ldots, a_d \in \R_{+}} t(a_0,...,a_d).\] 

By rescaling $f$, which does not change the zeros of the optimal linear approximant of $f$, we may assume $a_0=1$. We will apply the Lagrange multiplier method to this positive real finite-dimensional optimization problem.

The Lagrange multiplier analysis yields the following recurrence relations for $t$ and the coefficients of $f$, a function for which the optimal $t$ is attained. 

\begin{Theorem}\label{basic1}
Let $1 < p < \infty$, with $p \neq 2$, and let $d \in \N$ with $d>2$. Then:
\begin{enumerate}
\item $T_{d,p} $ is attained for some 
$$f(z) = \sum_{k=0}^d a_k z^k \in  \mathscr{P}_d$$ whose coefficients satisfy the recurrence relation
   \begin{equation}\label{system1}
( pa_k t - a_{k+1}) |a_{k+1} - t a_k|^{p-2} = (p-1) a_{k-1} |a_k - t a_{k-1}|^{p-2}, \end{equation}
 where $t=T_{d,p} $ and we understand that $a_{-1}=a_{d+1}=0$.
\item For that choice of $f$, $T_{d,p} $ is the unique solution to the equation 
 \begin{equation}\label{system2}h'(T_{d,p} )=0,\end{equation} where $h(s) = \|(1-sz)f(z)\|_p^p$.
\end{enumerate}  
\end{Theorem}

\begin{proof}
Since the coefficients of $f$ are assumed to be real and nonnegative (Proposition \ref{propo202}), we know that $h'(t)=0$ when the linear approximant for $f$ is of the form $c(1-tz),$ where $c>0$ and $t>0$.
The problem of determining $T_{d,p} $ is equivalent to maximizing $t$ subject to the restrictions that $h'(t)=0$, for $t, a_0,...,a_d \in \R_+$.  We apply the Lagrange multiplier method to this maximization  problem and we obtain 
\begin{equation}\label{lagrangerocksman}
\nabla (t) - \lambda \nabla(h'(t))=0,
\end{equation}
where the gradient $\nabla$ is taken with respect to the $(d+2)$-tuple of variables $(t, a_0, a_1,\ldots, a_d)$.

The first term in the left hand side of \eqref{lagrangerocksman}, $\nabla (t)$,  is just the constant vector $(1,0,...,0)$.  The first scalar equation in the system \eqref{lagrangerocksman} will merely determine the value of $\lambda$, which will not be of any use. The remaining scalar equations arising from \eqref{lagrangerocksman} are of the form \[\frac{\partial}{\partial a_j} h'(t) =0\] where $j=0,...,d$.
The result comes from calculating 
\[\frac{\partial}{\partial a_j} \frac{h'(t)}{p} = |ta_j -a_{j+1}|^{p-2} (pt a_j - a_{j+1}) - a_{j-1} (p-1) |ta_{j-1}-a_j|^{p-2}\]
and setting this equal to zero.
\end{proof}

Let us abbreviate by $\mathscr{L}_d$ the set of recurrence relations \eqref{system1}.

The recurrence relations in Theorem \ref{basic1} include the requirement
\begin{equation}\label{a2nd}
a_1 = t a_0 \ \ \mbox{or} \ \ a_1 = pta_0. \end{equation}

For the moment assume that the first of the choices is made in \eqref{a2nd}. Consider the polynomial $F(z) := b + zf(z)$, which has coefficient sequence $b, a_0 , a_1, a_2,\ldots, a_d$.  If we insist that \[a_0 = tb,\] then $F$ satisfies the recurrence relations for $k=0$, since we are essentially adding the term $0 = p(a_0 - tb)^{p-1}b$ to the right side.
Similarly, by definition, $F$ satisfies condition \eqref{a2nd}, in the form $a_0 = tb$.
Lastly, the recurrence relations remain true for the indices $k$, $2\leq k \leq d+1$. Thus $F$ also has an optimal linear approximant having a root at $1/t$ ($F$ is not necessarily extremal since the degree of $F$ is $d+1$).
Applying this observation repeatedly yields the following result.
\begin{Proposition}\label{propo78}
Let $1 < p < \infty$, with $p \neq 2$, and let $m, n \in \N$. Suppose that the real polynomial $f$ of degree $d$ satisfies the conditions \eqref{system1} for some value of $t$. Then so does the polynomial
\[
a_0 + a_0 tz + a_0 t^2 z^2 + \cdots + a_0 t^{m-1} z^{m-1} + t^m z^m f(z).
\]
\end{Proposition}

If $a_0=0$ the linear approximant would be identically zero. Thus, without loss of generality we could assume $a_0=1$. 
In other words, in order for a $(d+2)$-tuple $(t, a_0, a_1,\ldots, a_d)$ to solve the optimization problem as described in Theorem \ref{basic1}, it must satisfy the recurrence relations
\begin{align*}
a_0 &=1\\
a_1 &= t \mbox{\ or\ } pt \\
( pa_k t - a_{k+1}) |a_{k+1} - t a_k|^{p-2} &= (p-1) a_{k-1} |a_k - ta_{k-1}|^{p-2},\ 1\leq k < d \\ 
pa_dt | - t a_d|^{p-2} &= (p-1) a_{d-1} |a_d - ta_{d-1}|^{p-2}
\end{align*}
as well as the constraint $h'(t) = 0$, where 
$$h(t) = h_{f,p}(t) := \|(1 - tz)f(z)\|_p^p$$
and $f(z)=\sum_{k=0}^d a_k z^k$.

\begin{Proposition}\label{leadcoeffnonvan}
Let $1 < p < \infty$, with $p\neq 2$, and $d \in \N$ with $d > 2$.  If $f \in \mathscr{P}_d$  satisfies $t_{f} = T_{d,p} $,  then $a_d \neq 0$. 
\end{Proposition}

\begin{proof}
If $a_d = 0$, then by working backwards from the recurrence relations, we find that $a_{d-1}=0$, $a_{d-2}=0$, etc., concluding with $a_1=0$, a contradiction.
\end{proof}

From this we may obtain another important property of the optimal $t$. From the definition of $T_{d,p}$ we clearly have $T_{d + 1,p} \geq T_{d,p}$. This next proposition gives us something stronger.  

\begin{Proposition}\label{monoincr}
For  $1 < p < \infty$, with $p\neq 2$, the quantity $T_{d,p} $ is strictly increasing in $d$.
\end{Proposition}

\begin{proof}
There is a $(d+1)$-tuple $(T_{d-1,p}, a_0, a_1,\ldots, a_{d-1})$ that satisfies $\mathcal{L}_{d-1}$. Suppose, for the sake of argument, that $T_{d-1,p} = T_{d,p} $. Then the $(d+2)$-tuple $(T_{d-1,p}, a_0, a_1,\ldots, a_{d-1}, 0)$
must satisfy $\mathcal{L}_d$, in violation of Proposition \ref{leadcoeffnonvan}.
\end{proof}

From this last result, together with Proposition \ref{propo78}, it turns out that for the polynomial $f$ to correspond to an optimal $T_{d,p} $, we need to make the selection $a_1 = a_0 pt$.  We can exploit this and it leads to further reductions.

\begin{Proposition}
Let  $1 < p < \infty$, with $p\neq 2$ and let $d \in \N$ with $d > 2$.  To find $T_{d,p} $ by means of the recurrence relations $\mathscr{L}_d$, it suffices to consider polynomials 
$$f(z) = \sum_{k = 0}^{d} a_k z^k \in \mathscr{P}_d$$
such that $a_0 = 1$, $a_1 = pt$, and all of the coefficients $a_2, a_3,\ldots, a_d$ are positive.
\end{Proposition}

\begin{proof}
Suppose that the $(d+2)$-tuple $(t, a_0, a_1,\ldots, a_d)$ solves $\mathcal{L}_d$.   We have already shown that the assumption $a_0 =1$ is no loss of generality, and that the selection of $a_1 = pta_0$ is necessary for optimality.

Next, suppose that $a_k = 0$ for some $k$, $2 \leq k < d$. By inspection of the resulting recurrence relation
\[
pa_{k-1} t |- t a_{k-1}|^{p-2} = (p-1) a_{k-2} |a_{k-1} - ta_{k-2}|^{p-2},
\] 
we find that the $(k+1)$-tuple $(T_{d,p} , a_0, a_1,\ldots, a_{k-1})$ solves the recurrence relations of   $\mathcal{L}_{k-1}$. Furthermore, the constraint equation $h'(t) = 0$ holds for the polynomial $a_0 + a_1 z + \cdots + a_{k-1} z^{k-1}$, since the assumption $a_1 = pt$ was made. Therefore 
$(T_{d,p} , a_0, a_1,\ldots, a_{k-1})$ solves all the requirements of $\mathcal{L}_{k-1}$. This contradicts Proposition \ref{monoincr}.
\end{proof}

There is yet a further reduction in the scope of the problem.  Recall that solving the Lagrange multiplier problem entails satisfying all of the recurrence relations $\mathscr{L}_d$, along with the constraint equation $h'(t)=0$, where we recall that
\begin{equation}\label{firstthree}
      h'(t)  = - \sum_{k=1}^{d} p(a_k - ta_{k-1})^{\langle p-1 \rangle}a_{k-1} + pt^{\langle p-1\rangle}|a_d|^p.
\end{equation}
It turns out that 
the condition $h'(t) = 0$  in Theorem \ref{basic1} is automatic if $a_1 = pt$ (which, by the previous Proposition, we may assume).

\begin{Proposition}\label{cubiccaseb}
Let  $1 < p < \infty$, with $p\neq 2$ and let $d \in \N$ with $d > 2$. Suppose $f(z) = a_0 + a_1 z + a_2 z^2 +\cdots+  a_d z^d$ satisfies the equations 
     \begin{align}
      a_0 &= 1 \nonumber\\
      a_1 &=  pt  \nonumber \\
      ( pa_k t - a_{k+1}) |a_{k+1} - t a_k|^{p-2} &= (p-1) a_{k-1} |a_k - ta_{k-1}|^{p-2},\ 1\leq k < d \label{boxcarb} \\ 
       pa_dt |- t a_d|^{p-2} &= (p-1) a_{d-1} |a_d - ta_{d-1}|^{p-2}. \label{cabooseb}
\end{align}
Then $h'(t) = 0$, where $h(t) := \|(1 - tz)f(z)\|_p^p$.
\end{Proposition}

\begin{proof}  

We begin with the first three terms of the formula for $h'(t)$ from \eqref{firstthree}.  Use substitution and simplification to get
\begin{align}
      ( p^2 t^2- a_{2}) |a_{2} - pt^2|^{p-2} &= (p-1)^{p-1} t^{p-2}\label{f2ndb}\\ 
            ( pa_2 t - a_{3}) |a_{3} - t a_2|^{p-2} &= (p-1) pt |a_2 - pt^2|^{p-2}\label{f3rdb}\\ 
       pa_3^{\langle p-1 \rangle} t ^{p-1} &= (p-1) a_{2} |a_3 - ta_{2}|^{p-2}\label{f4thb}
\end{align}
recalling that $t>0$.

Straightforward calculation yields
\begin{align}
  &\qquad -\sum_{k=1}^3 (a_k - t a_{k-1})^{\langle p-1\rangle} a_{k-1} \\
   &= -(p-1)^{p-1} t^{p-1} - (a_2 - pt^2)^{\langle p-1 \rangle} pt - (a_3 - ta_2)^{\langle p-1\rangle} a_2  \nonumber\\
       &= -(p-1)^{p-1} t^{p-1} - |a_2 - pt^2|^{p-2}(a_2 - pt^2) pt\nonumber\\
       & \qquad  - |a_3 - ta_2|^{p-2} (a_3 - ta_2)a_2.
     \label{hprimeb}
\end{align}
Use \eqref{f2ndb} to rewrite the first term in \eqref{hprimeb}.  The result is
\begin{align}
   {\ } &= -t(p^2 t^2 - a_2)|a_2-pt^2|^{p-2} - (a_2 -pt^2)^{\langle p-1\rangle}pt\nonumber\\
     &\quad  - (a_3-ta_2)^{\langle p-1 \rangle }a_2 \nonumber\\
     &= -(p-1)a_2 t|a_2 - pt^2|^{p-2}  - (a_3-ta_2)^{\langle p-1 \rangle }a_2 \label{monkeybarsb}.
\end{align}
Next, use \eqref{f3rdb} to rewrite the first term of \eqref{monkeybarsb} with the result
\begin{align*}
     {\ } &=  -(p-1)a_2 t \cdot\frac{1}{(p-1)pt} (pa_2 t - a_3)|a_3 - ta_2|^{p-2} - (a_3-ta_2)^{\langle p-1 \rangle }a_2  \\
          &= [-(1/p)(pa_2 t - a_3) ]  |a_3 -ta_2|^{p-2}a_2 - (a_3-ta_2)^{\langle p-1 \rangle }a_2\\
          &=  -\frac{p-1}{p} a_2 a_3 |a_3 - t a_2|^{p-2}.
\end{align*}

We have just established that
\[
      -\sum_{k=1}^3 (a_k - t a_{k-1})^{\langle p-1\rangle} a_{k-1} = -\frac{p-1}{p} a_2 a_3 |a_3 - t a_2|^{p-2}.
\]
By \eqref{boxcarb} with $k=3$, we have
\begin{align*}
    &  -\sum_{k=1}^3 (a_k - t a_{k-1})^{\langle p-1\rangle} a_{k-1}\\ &= -\frac{p-1}{p} a_2 a_3 |a_3 - t a_2| ^{p-2} \\
               &=  -\frac{1}{p} a_3 (pa_3 t - a_4)|a_4 - t a_3|^{p-2}\\
               &=  (-a_3^2 t + a_3 a_4 /p) |a_4 - t a_3|^{p-2}.
\end{align*}

From this it follows that 
\begin{align*}
    &  -\sum_{k=1}^4 (a_k - t a_{k-1})^{\langle p-1\rangle} a_{k-1}\\
    &= (-a_3^2 t + a_3 a_4 /p) (a_4 - t a_3)^{p-2} - (a_4 - t a_3)^{\langle p-1\rangle} a_3\\ 
      &= \big[  -a_3^2 t + a_3 a_4 /p  - a_3 a_4 +a_3^2 t  \big] |a_4 - t a_3|^{p-2}  \\
      &= - \frac{p-1}{p} a_3 a_4 |a_4 - t a_3|^{p-2}.
\end{align*}
Repeating this argument leads to
\[
           -\sum_{k=1}^d (a_k - t a_{k-1})^{\langle p-1\rangle} a_{k-1} =   - \frac{p-1}{p} a_{d-1} a_d |a_d - t a_{d-1}|^{p-2}.
\]
Finally, from \eqref{cabooseb} we find that the last term of $h'(t)$ is
\[
       t^{p-1}|a_d|^p = \frac{p-1}{p} a_{d-1} a_d |a_d - ta_{d-1}|^{p-2}.
\]
  In conclusion, $h'(t) = 0$.
\end{proof}

This shows that if we make the selection $a_1 = a_0 p t$ (rather than the other choice, $a_1 = a_0 t$), then the recurrence relations already imply that $h'(t) = 0$. 

Theorem \ref{basic1}  gives us a method for calculating $T_{d,p} $ and finding one of infinitely many polynomials $f$ for which it is attained.  Here is a list of solutions obtained by using Mathematica to solve $\mathscr{L}_d$ numerically.  Here we chose $p$ to be an even integer, since in that case the numerical calculations converge much more quickly.  We found that in all cases, except when $d=2$ and $p=4$, there is an extra zero $1/t$ inside the disk $\mathbb{D}$.

\begin{tabular}{|c|c|c|c|}
\hline$d$ & $p$ & $1/t$ & $f(z)$ \\ \hline
$2$ & $4$ & 
$1.09638$ & 
$1 + 3.64836z + 1.92310z^2$\\ \hline
$2$ &
$6$ & 
$0.95629$ & 
$1 + 6.27424z + 3.36907z^2$\\ \hline
$2$ & 
$8$ & 
$0.88193$ & 
$1 + 9.07101z + 4.96676z^2$\\ \hline
$2$ & 
$10$ & 
$0.83568$ & 
$1 + 11.9663z + 6.65305z^2$\\ \hline
 %
 %
 %
 %
%
 %
$3$ & 
$4$ & 
$0.94921$ & 
$1 + 4.21406z + 3.01393z^2 + 1.65036z^3$\\ \hline
$3$ & 
$6$ & 
$0.82606$ & 
$1 + 7.26338z + 5.34352z^2 + 3.00715z^3$\\ \hline
$3$ & 
$8$ & 
$0.76236$ & 
$1 + 10.4938z +7.89188 z^2 + 4.54074z^3$\\ \hline
$ 3$ & 
$10$ & 
$0.72322$ & 
$1 + 13.8270z + 10.57437z^2 + 6.18409z^3$\\ \hline
%
%
$4$ & 
$4$ & 
$0.89213$ & 
$1 + 4.48365z + 3.59236z^2 + 2.59647z^3 + 1.44035z^4$\\ \hline
$4$ & 
$6$ & 
$0.77760$ & 
$1 + 7.71608z + 6.35232z^2 + 4.74328z^3 + 2.71501z^4$\\ \hline
$4$ & 
$8$ & 
$0.71878$ & 
$1 + 11.13000z + 9.37221z^2 + 7.14758z^3 + 4.18719z^4$\\ \hline
$4$ & 
$10$ & 
$0.68277$ & 
$1 + 14.6463z + 12.51665z^2 + 9.69994z^3 + 5.77764z^4$\\  \hline
%
%
%
%
%
\end{tabular}


\section{A dynamical systems approach}\label{Sect7}

By definition,  $z_0 \in \Omega_p$ if the optimal polynomial for some non-constant $f \in \ell^{p}_{A}$ vanishes at $z_0$. The results of the previous sections have shown that $\tau_p \in (1, 2)$ and either 
$$\Omega_p = \C \backslash\textstyle{\frac{1}{\tau_{p}}} \D, \quad \text{or} \quad \Omega_p = \C \backslash\textstyle{\frac{1}{\tau_{p}}} \overline{\D},$$ 
where 
$$\tau_{p} = \lim_{d \to \infty} T_{d, p}.$$
Recall that the limit above exists since $T_{d}$ is bounded and strictly increasing. As mentioned, when $p = 2$, $\Omega_2 = \C \backslash\overline{\D}$. We know from previous sections that when $p \neq 2$, $1<\tau_p<2$. In the present section, we will calculate $\tau_{p}$ by means of a dynamical system derived from the recurrence relations of Theorem \ref{basic1}.   The qualitative nature of the dynamical system depends on whether $1<p<2$ or $2<p<\infty$.  Details will be provided for the latter case, and the former can be handled in an analogous way.


Fix $p > 2$. We transform the recurrence relations \eqref{system1} for the coefficients $\{a_k\}$ into recurrence relations for the ratios 
$$R_{k} = \frac{a_k}{a_{k-1}}, \quad k\in \mathbb{N}.$$
These take the form
\begin{equation}\label{recrelrk}
(pt - R_{k+1})|R_{k+1}-t|^{p-2} = (p-1)\frac{1}{R_k}\Big|1 - \frac{t}{R_k}\Big|^{p-2},
\end{equation}
$1 \leq k < d$.
The equation \eqref{recrelrk} can equivalently be expressed as
\begin{equation}\label{phipsirel}
\Phi(R_{k+1}) = \Psi(R_k),\quad 1 \leq k < d,
\end{equation}
where
$$
\Phi(x) := (pt - x)|x-t|^{p-2}$$
and 
$$\Psi(x):= (p-1)\frac{1}{x}\Big|1 - \frac{t}{x}\Big|^{p-2}.
$$
See Figure \ref{kksdf}.
Let us 
 stress that $\Phi$ and $\Psi$ depend on the values of $p$ and $t$.
If we are able to find $R_1$, $R_2$,\ldots, $R_d$ and $t$ satisfying \eqref{phipsirel}, then by taking $a_0 = 1$ we obtain a solution $a_0$, $a_1$, $a_2$,\ldots, $a_d$ and $t$ for \eqref{system1}.

Our plan is to implement the relation \eqref{phipsirel} graphically, along with the requirements that 
$$R_1 = pt \quad \mbox{and} \quad R_{d+1} =0.$$  This gives rise to an implicit dynamical system with two parameters, $p$ and $t$.  By analyzing this dynamical system we will be able to identify solutions to the recurrence relations \eqref{recrelrk} for the ratios $R_k$, and discern which solutions correspond to optimal values of $t$. 

%

%
The following lemmas concerning the structure of $\Phi$ and $\Psi$ are easy to verify from basic calculus.
\begin{Lemma}\label{firstlem}
Let $p>2$ and $t \geq 1$.
Then
\begin{enumerate}
\item $\Phi$ is decreasing on the intervals $(-\infty,t]$ and $[(p-1)t, \infty)$ and increasing on the interval $[t,(p-1)t]$;
\item $\Phi$ intercepts the axes at the points $(0, pt^{p-1})$, $(t,0)$ and $(pt,0)$;
\item $\Phi$ attains a local minimum at the point $(t,0)$ and a local maximum at the point $((p-1)t, (p-2)^{p-2}t^{p-1})$.
\end{enumerate}
\end{Lemma}

\begin{Lemma}
Let $p>2$ and $t \geq 1$.
Then
\begin{enumerate}
\item $\Psi(x) \geq 0$ if and only if $x>0$, with equality precisely when $x=t$;
\item $\Psi$ is decreasing on the intervals $(-\infty,0)$, $(0,t)$ and $[(p-1)t, \infty)$, and increasing on the interval $[t,(p-1)t]$;
\item $\Psi$ has a vertical asymptote at $x=0$, with $\lim_{x\rightarrow0-}\Psi(x) = -\infty$ and $\lim_{x\rightarrow0+}\Psi(x) = +\infty$, and a horizontal asymptote at $y=0$, with $\lim_{x\rightarrow \pm \infty}\Psi(x) = 0$;
\item $\Psi$ has an $x$-intercept at the point $(t,0)$;
\item $\Psi$ attains a local minimum at the point $(t,0)$, and a local maximum at $\Big((p-1)t, \frac{1}{t}\big(\frac{p-2}{p-1}\big)^{p-2}\Big)$.
\end{enumerate}
\end{Lemma}

In addition, the graphs of the functions $\Phi$ and $\Psi$ interact in the following ways. Figure \ref{kksdf4} captures some of these effects.

\begin{Lemma}\label{3rdlem}
Let $p>2$ and $t \geq 1$.
Then
\begin{enumerate}
\item If $t>1$, there exists a neighborhood of $t$ in which $\Phi(x) \geq \Psi(x)$, with equality precisely at $x=t$;
\item If $t=1$, then there exists a neighborhood of $t$ in which $\Phi(x) >\Psi(x)$ when $x>t$, and $\Phi(x) < \Psi(x)$ when $x<t$;
\item If $t>1$, then there are exactly three values of $x$ for which $\Phi(x) = \Psi(x)$, including $x=t$, with $x=t$ lying between the other two such points (we will call these ``fixed points'');
\item If $t=1$, then there are exactly two values of $x$ for which $\Phi(x) = \Psi(x)$, including $x=t=1$, with the other such point being greater;
\item As the parameter $t$ increases, $\Psi(pt)$ decreases, and $\Phi(\xi_1)$ increases, where $\xi_1$ is the least of the fixed points.
\end{enumerate}
\end{Lemma}


\begin{proof}
Assertions (a) and (b) follow from substituting $x = t + \epsilon$ into $\Phi$ and $\Psi$, and finding that for small $\epsilon$, we have
\begin{align*}
\Phi(t+\epsilon) &\approx (p-1)t\epsilon^{p-2}\\
\Psi(t+ \epsilon) &\approx (p-1)\frac{\epsilon^{p-2}}{t^{p-1}}.
\end{align*}
In order to establish (c), based on (a) and the behavior of $\Phi$ and $\Psi$ near $x=0$, it is elementary that there are at least one fixed point in $(0,t)$, one at $t$ and one between $t$ and $pt$. If $x \neq t$, the fixed point condition can be simplified from
\[
(pt-x)|x-t|^{p-2} = (p-1)\frac{1}{x}\Big|1 - \frac{t}{x}\Big|^{p-2}\]
to the much simpler
\[g(x):=x^{p}-ptx^{p-1}+ (p-1) =0.\]
Since the function $g$ has only one change of direction for $x>0$, there cannot be more than 2 fixed points apart from the one at $x=t$.
We can make the same simplifications in order to derive (d) and (e) (we already know that $x=t$ is a fixed point).  This time, when $t=1$, one of the solutions to the simplified equation is still at $x=t$.

Next, consider the graphs of the two equations
\[
y = pt-x\ \ \mbox{and} \ \ y = \frac{p-1}{x^{p-1}}.
\]
They intersect at two points; as $t$ increases, the rightmost intersection point $\xi_2$ tends to the right, while the leftmost intersection point $\xi_1$ tends to the left. It is elementary to see that $\Psi(pt)$ decreases as $t$ increases; $\Phi(\xi_1)$ increases with $t$ due to Lemma \ref{firstlem}, part (a). 
\end{proof}
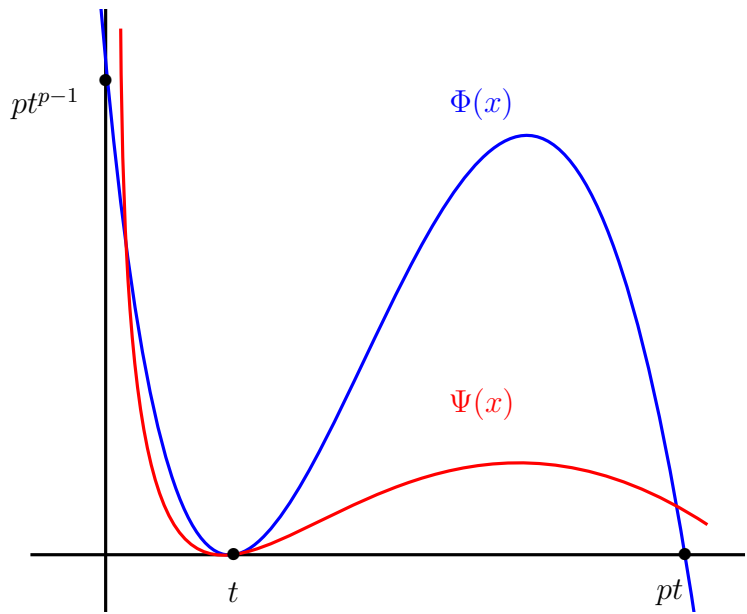
\begin{figure}[h]
\centering
\begin{tikzpicture}[scale = 1.0]
\clip (-1.5,-0.75) rectangle + (10,8);
\draw[very thick, black] (-1, 0) -- (8.5, 0);
\draw[very thick, black] (0, -2) -- (0, 8);
 \draw[very thick, blue] (-0.2, 9) .. controls (1.58, -15.4) and (5, 20) ..  (8, -2);
  \draw[very thick, red] (0.2, 7) .. controls (0.3, -6.2) and (2.5, 4) ..  (8, 0.4);
  \node at (5,6) {${\color{blue}\Phi(x)}$};
    \node at (5,2) {${\color{red}\Psi(x)}$};
    \node [black] at (1.7,0) {\textbullet};
        \node [black] at (0, 6.3) {\textbullet};
        \node at (1.7,-.5) {$t$};
          \node at (-0.8, 6) {$p t^{p - 1}$};
            \node [black] at (7.7,0) {\textbullet};
               \node at (7.5,-.5) {$pt$};
        \end{tikzpicture}
\caption{The graphs of $\Phi(x)$ and $\Psi(x)$.}
\label{kksdf}
\end{figure}
Figure \ref{kksdf} illustrates some of the important features of the functions $\Phi$ and $\Psi$.  
For instance, the graphs intersect at exactly three points.  One of these points is always $(t,0)$.  The other two intersection points are especially important in what follows.  The $x$-intercept of $\Phi$ at $(pt,0)$ plays a role, reflecting that $R_1 = pt$.  Similarly, the $y$-intercept of $\Phi$ at $(0,pt^{p-1})$ corresponds to the desired condition $R_{d+1}=0$.  How these points relate, and how the entire picture varies with $t$, will be relevant to the analysis below.

\begin{Remark}  All of the figures appearing in this paper were created out of Bezier curves using Tikz.   They are intended to illustrate the qualitative behavior of the functions, and present a graphical solution to the Lagrange multiplier problem.  They are not, however, the actual graphs of $\Phi$ and $\Psi$; this is because there is no linear scale at which the critical features of mathematically accurate graphs could be discerned by the human eye, and a nonlinear scale would distort the straight line segments that will play a central part in the solution.
\end{Remark}

We now describe a dynamical system associated with $\Phi$ and $\Psi$.  Suppose that a point $x_1 \in \mathbb{R}$ is selected.  If this point represents the value of $R_1$, then $R_2$ can be determined by the equation $\Phi(R_2) = \Psi(x_1)$.  To do this graphically, we first draw a vertical line segment connecting the points $(x_1, 0)$ and $(x_1, \Psi(x_1))$.  Then we consider the horizontal line passing through the latter point.  This horizontal line must intersect the graph of $y = \Phi(x)$ at one, two or three points.  If $(x_2, \Phi(x_2))$ is such an intersection point, then this tells us $x_2$ is a candidate for the value of $R_2$.   Let us accordingly draw a horizontal line segment connecting the points $(x_1, \Psi(x_1))$ and $(x_2, \Phi(x_2))$.  By continuing in this fashion, we obtain a path consisting of  alternating vertical and horizontal line segments, intersecting the graphs of $\Psi$ and $\Phi$. Generally, there are multiple such paths, reflecting that there may be multiple candidates for $x_n$ at the $n$th step.   Each associated sequence of abscissas, $x_1$, $x_2$, $x_3$,\ldots, serves as a candidate for the sequence of values of the coefficient ratios $R_1$, $R_2$, $R_3$,\ldots.   These paths are called the ``orbits'' or ``trajectories'' of the initial point $x_1$.  Of course we are principally interested in choosing $x_1 = pt$.

The vast majority of these orbits do not correspond to valid solutions of the Lagrange multiplier problem for extremal functions $f$.  This is because they do not ``terminate,'' or to be more precise, there is no value of $d \in \mathbb{N}$ for which $x_d = 0$.  It is easy to see that there are orbits that converge to the fixed points of the system, that is, the points $\xi \in \mathbb{R}$ at which $\Phi(\xi) = \Psi(\xi)$.   In the proof of Lemma \ref{3rdlem} we have already found that there are three fixed points $\xi_1$, $t$ and $\xi_2$, satisfying $0< \xi_1 < t < \xi_2$.   In order for 
an orbit to terminate, it must be that the orbit enters the region of the graph where $0 \leq x < \xi_1$.  It must further be the case that the orbit exactly reaches the point $(0, pt^{p-1})$, where the graph of $\Phi$ and the $y$-axis intersect.   In this way, a finite set of coefficients satisfying \eqref{system1} is produced. 
\begin{figure}[h]
\centering
\begin{tikzpicture}[scale = 1.0]
\clip (-0.5,-0.75) rectangle + (8.5,8);
\draw[very thick, black] (-1, 0) -- (8.5, 0);
\draw[very thick, black] (0, -2) -- (0, 8);
 \draw[very thick, blue] (-0.2, 9) .. controls (1.58, -15.4) and (5, 20) ..  (8, -2);
  \draw[very thick, red] (0.2, 7) .. controls (0.3, -6.2) and (2.5, 4) ..  (8, 0.4);
  \node at (5,6) {${\color{blue}\Phi(x)}$};
    \node at (5,2) {${\color{red}\Psi(x)}$};
    \node [black] at (1.7,0) {\textbullet};
        \node at (1.7,-.5) { $t$};
            \node [black] at (7.7,0) {\textbullet};
               \node at (7.1,-.3) {{\tiny $R_1 = pt$}};
               \draw[green, very thick] (7.68, 0) -- (7.66, 0.66);
                \node [green] at (7.66,0.66) {\textbullet};
               \draw[green, very thick] (1, 0.7) -- (7.66, 0.66);
                \node [green] at (1,0.7) {\textbullet};
                 \draw[green, very thick] (1, 0.7) -- (1, 0.3);
                    \node at (0.9, -.3) {{\tiny $R_2$}};
                   \node [black] at (1,0) {\textbullet};
                  \node [green] at (1,0.3) {\textbullet};
                   \draw[green, very thick] (1, 0.3) -- (1.2, 0.3);
                   \node [green] at (1.2,0.3) {\textbullet};
                    \node [black] at (1.2,0) {\textbullet};
                     \node at (1.25, -.3) {{\tiny $R_3$}};
                      \node at (1.55, -.3) {...};
                    \draw[green, very thick] (1.2, 0.3) -- (1.2, 0.1);
                    \node [green] at (1.2,0.1) {\textbullet};
                      \draw[green, very thick] (1.2, 0.1) -- (1.4, 0.1);
                      \node [green] at (1.4,0.1) {\textbullet};
                      \draw[green, very thick] (1.4, 0.1) -- (1.4, 0.01);
                        \node [green] at (1.4,0.01) {\textbullet};
                        \draw[green, very thick] (1.4, 0.01) -- (1.55, 0.01);
                         \node [green] at (1.55,0.01) {\textbullet};
        \end{tikzpicture}
\caption{The graphs of $\Phi(x)$ and $\Psi(x)$ with an orbit which does not terminate with $R_{d} = 0$ for some $d$.}
\label{kksdf2}
\end{figure}
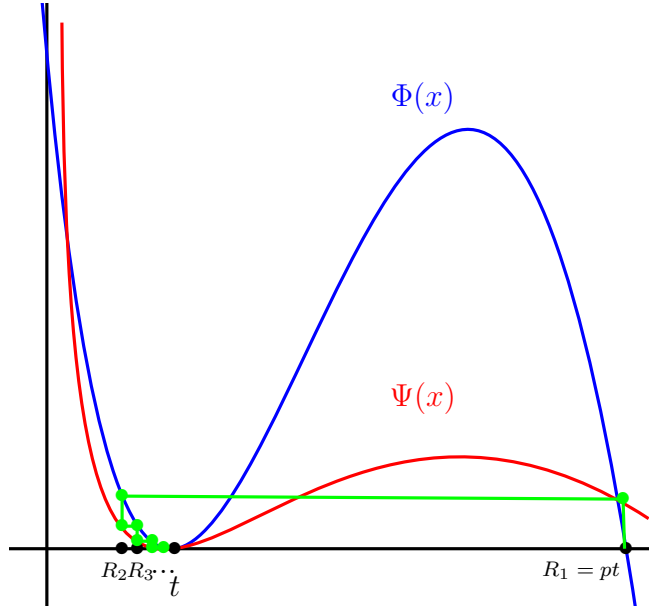

In Figure \ref{kksdf2}, the orbit is destined to converge to the point $(t,0)$, and cannot correspond to a solution to the recurrence relations.  In fact, since the output from the leftmost fixed point is higher than that of the rightmost fixed point, there can be no orbit originating from $(pt,0)$ that can reach the exit point $(0, pt^{p-1})$.  Accordingly, there is no solution to the Lagrange system attached to this value of $t$.

\begin{figure}[h]
\centering
\begin{tikzpicture}[scale = 1.0]
\clip (-0.7,-0.75) rectangle + (8.5,8);
\draw[very thick, black] (-1, 0) -- (8.5, 0);
\draw[very thick, black] (0, -2) -- (0, 8);
 \draw[very thick, blue] (-0.26, 9) .. controls (1.4, -15.3) and (5.4, 19.8) ..  (8, -2);
  \draw[very thick, red] (0.4, 7) .. controls (0.5, -6.4) and (2.5, 4) ..  (8, 1.9);
      \draw[green, very thick] (7.75, 0) -- (7.75, 2);
       \node [black] at (7.75, 0) {\textbullet};
         \node at (7.6, -.3) {{\tiny $R_1$}};
       \draw[green, very thick] (0.55,2) -- (7.75, 2);
       \node at (0.55, 0) {\textbullet};
         \node at (0.55, -.3) {{\tiny $R_2$}};
          \node [green] at (0.4,2.7) {\textbullet};
          \node [black] at (0.25, -0.3) {{\tiny $R_3$}};
           \draw[green, very thick] (0.55,2) -- (7.75, 2);
            \node [black] at (0.4, 0) {\textbullet};
             \node [green] at (7.75, 2) {\textbullet};
                \node [black] at (0, 0) {\textbullet};
                 \node [black] at (-0.2, -0.3) {{\tiny $R_4$}};
             \draw[green, very thick] (0.55,2) -- (0.55, 2.7);
              \node [green] at (0.55,2) {\textbullet};
              \draw[green, very thick] (0.4,2.7) -- (0.55, 2.7);
               \node [green] at (0.55,2.7) {\textbullet};
               \draw[green, very thick] (0.4,2.7) -- (0.4, 6);
               \node [green] at (0.4, 6) {\textbullet};
               \draw[green, very thick] (0, 6) -- (0.4, 6);
                \node [green] at (0, 6) {\textbullet};
                  \node [black] at (-0.37, 6) {{\tiny $p t^{p - 1}$}};
                    \node at (5,6) {${\color{blue}\Phi(x)}$};
    \node at (5,2.5) {${\color{red}\Psi(x)}$};
\end{tikzpicture}
\caption{A successful choice of $t$ that produces an orbit reaching the desired exit point.}
\label{kksdf3}
\end{figure}
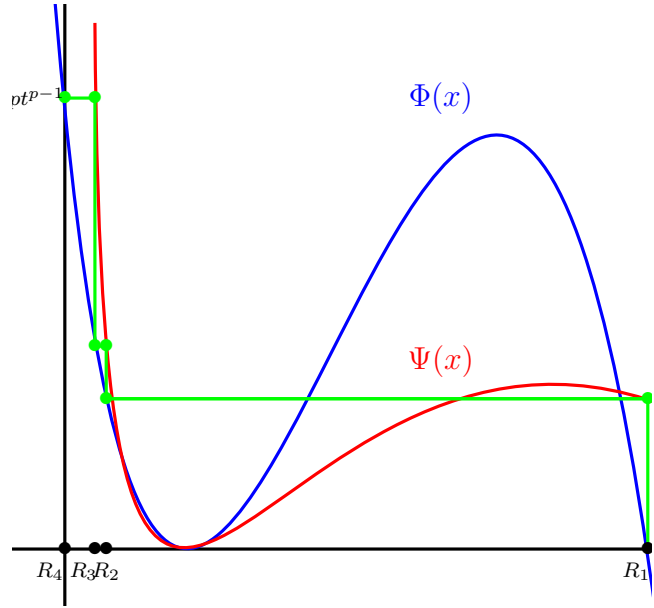

Figure \ref{kksdf3} presents a different situation, in which the choice of $t$ gives rise to an orbit that does hit the exit point. In fact, we can see that $R_4 = 0$, and hence the orbit corresponds to a cubic solution to the Lagrange multiplier problem.

\begin{figure}[h]
\centering
\begin{tikzpicture}[scale = 1.0]
\clip (-0.7,-0.75) rectangle + (9,8);
\draw[very thick, black] (-1, 0) -- (8.5, 0);
\draw[very thick, black] (0, -2) -- (0, 8);
 \draw[very thick, blue] (-0.26, 9) .. controls (1.4, -15.3) and (5.4, 19.8) ..  (8, -2);
  \draw[very thick, red] (0.4, 7) .. controls (0.5, -6.4) and (2.5, 4) ..  (8, 1.9);
                  \node [black] at (-0.35, 6) {{\tiny $p t^{p - 1}$}};
                     \node [black] at (1.5, -0.3) {$t$};
                      \node [black] at (8, -0.3) {$p t$};
                       \node [black] at (1.6, 0) {\textbullet};
                         \node [black] at (7.75, 0) {\textbullet};
                           \draw[magenta, very thick] (7.35, 0) -- (7.35, 2.1);
                            \node [black] at (7.35, 0) {\textbullet};
                             \node [green!40!red] at (7.35, -0.3) {$\xi_2$};
                              \draw[magenta, very thick] (0.8, 0) -- (0.8, 1.1);
                               \node [black] at (0.8, 0) {\textbullet};
                             \node [green!40!red] at (0.8, -0.3) {$\xi_1$};
                             \draw [decorate,
	decoration = {brace}] (0.65, 0) --  (0.65, 1);
	 \node [orange] at (0.05, 0.5 ) {inc w/$t$};
	  \draw [decorate,
	decoration = {brace}] (7.2, 0) -- (7.2, 2.1);
	 \node [orange] at (6.5, 1 ) {dec w/$t$};
	   \node at (5,6) {${\color{blue}\Phi(x)}$};
    \node at (5,2.5) {${\color{red}\Psi(x)}$};
\end{tikzpicture}
\caption{$\xi_1$ and $\xi_2$ are two of the three fixed points. The starting point is $(p t, 0)$ while the desired exit point is $(0, p t^{p - 1})$. }
\label{kksdf4}
\end{figure}
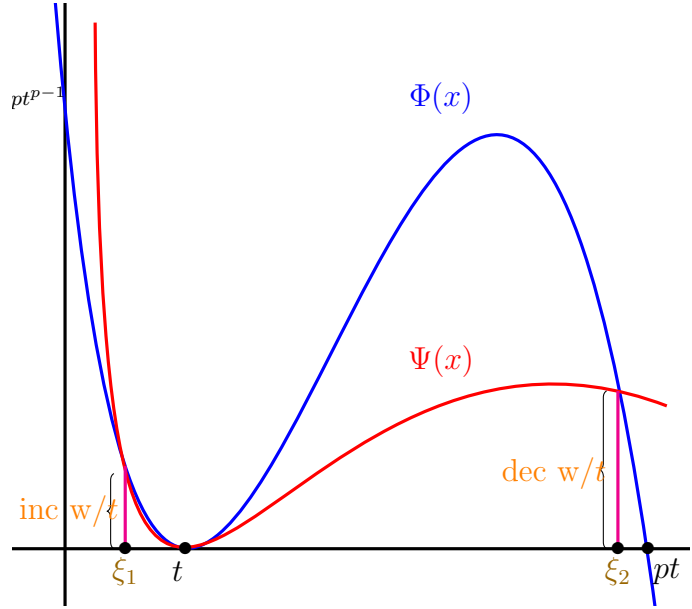

In both cases it is clear that the relative positioning of the points $(\xi_1, \Phi(\xi_1))$ and $(\xi_2, \Phi(\xi_2))$ has a significant effect on the eventual fate of the orbits.  In fact, when $t=1$,  $\Phi(\xi_1)$ starts out less than $\Phi(\xi_2)$; as $t$ increases, $\Phi(\xi_1)$ rises relative to  $\Phi(\xi_2)$, until they are equal at some unique value $t = \tau$.  Figure \ref{kksdf4} captures some of these effects.

The following theorem provides a means to calculate $\tau_p$, and hence the radius $1/\tau_p$ of the largest disk excluded from $\Omega_p$.
To prove this theorem, we need to look at the pre-orbits of the exit point $(0, pt^{p-1})$.

\begin{Theorem}\label{lastth}
$1 < p < \infty$, with $p\neq 2$, and for $t>1$, let $\xi_1=\xi_1(t)$ and $\xi_2=\xi_2(t)$ denote  the two distinct positive solutions to the equation
\[
x^{p-1}(pt-x) = p-1,
\]
with $\xi_1<\xi_2$.
Then there is a unique value $\tau$ of $t$ for which
\[
(pt-\xi_1)|\xi_1 - t|^{p-2} = (pt-\xi_2)|\xi_2 - t|^{p-2},
\]
and it satisfies
\[
  \tau = \tau_p =  \lim_{d\rightarrow\infty} T_{d,p} .
\]
Moreover, there is no solution in $\ell^p_A$ to the recurrence relations with $t=\tau_p$.
\end{Theorem}

\begin{proof}  We proceed under the assumption that $p>2$.  The case when $1<p<2$ is similar, but the graphs of $\Phi$ and $\Psi$ have a different shape. 

Previously we have seen that as $t$ increases from the value 1, $\Phi(\xi_1)$ increases relative to $\Psi(\xi_2)$.  Therefore, for some value $t = \tau$, we have $\Phi(\xi_1) = \Psi(\xi_2)$.  There is also a value $t = T$ such that $\Psi(pT) = \Phi(\xi_1)$.  It is clear that if $t \leq T$, then no orbit originating from $(pt, 0)$ will reach the exit point $(0, pt^{p-1})$: Indeed, otherwise the orbit  must converge towards the fixed point at $x=t$.

For any $T< t<\tau$ fixed, there are infinitely many pre-orbits of $(0, pt^{p-1})$ such that $\Phi(x_m) < \Psi(\xi_2)$ for some $x_m$ satisfying $0 < x_m < \xi_1$.   In between the fixed points these paths tend toward the horizontal line $y = \Phi(\xi_1)$.  

Of course these pre-orbits generally need not strike the point $(pt, 0)$, which would signal a candidate solution to the Lagrange system \eqref{system1}.  

However, the entirety of the graph of $\Phi$ and $\Psi$ and the associated orbits (viewed as a subset of the Euclidean plane) is a continuous object in the parameter $t$ for $1< t<\tau_p$.  
As $t$ increases, the collection of origination points for the pre-orbits of $(0, pt^{p-1})$ are continuously swept toward the left, relative to the point $(pt, 0)$.
Thus as $t$ increases toward $\tau_p$, there are countably many values of $t$ for which the pre-orbit of the exit point $(0, pt^{p-1})$ originates from $(0, pt)$.  The corresponding collections of ratios $R_1$, $R_2$,\ldots, $R_d$ gives rise to a solution to \eqref{system1} for this particular value of $t$.     In fact, by the continuity of the graph relative to $t$, there is such a solution for all $d$ from some integer onward. 

This situation is illustrated in Figure \ref{kksdf5}.  Pre-orbits of the exit point are shown in various shades of green.  As $d$ increases, they get ever closer to the orange path, which traces back to the condition $y = \Phi(\xi_1)$.   As $t$ increases, the $x$-intercepts of the green paths are dragged toward the left where in succession they hit the point $x = pt$.  In this way, we obtain solutions to the Lagrange system for the optimal coefficients, and the corresponding values of $t$, for all $d$ sufficiently large.

\begin{figure}[h]
\centering
\begin{tikzpicture}[scale = 1.0]
\clip (-0.7,-0.75) rectangle + (11,8);
\draw[very thick, black] (-1, 0) -- (11, 0);
\draw[very thick, black] (0, -2) -- (0, 11);
 \draw[very thick, blue] (-0.26, 9) .. controls (1.4, -15.3) and (5.4, 19.8) ..  (8, -2);
  \draw[very thick, red] (0.4, 7) .. controls (0.5, -6.4) and (2.5, 4) ..  (8, 1.9);
      \draw[green, very thick] (7.75, 0) -- (7.75, 2);
       \node [black] at (7.75, 0) {\textbullet};
       \draw[green, very thick] (0.55,2) -- (7.75, 2);
           \draw[green, very thick] (0.55,2) -- (7.75, 2);
             \draw[green, very thick] (0.55,2) -- (0.55, 2.7);
              \draw[green, very thick] (0.4,2.7) -- (0.55, 2.7);
               \draw[green, very thick] (0.4,2.7) -- (0.4, 6);
               \draw[green, very thick] (0, 6) -- (0.4, 6);
                  \node [black] at (-0.35, 6) {{\tiny $p t^{p - 1}$}};
                    \node at (5,6) {${\color{blue}\Phi(x)}$};
    \node at (5,2.5) {${\color{red}\Psi(x)}$};
      \draw[very thick, red] (0.4, 7) --  (0.22, 10);
        \draw[very thick, red] (8, 1.9) --  (11.6, 0.001);
      \draw[very thick, orange] (9.3, 1.2) -- (0.8, 1.2);
        \draw[very thick, orange] (9.3, 1.2) -- (9.3, 0);
           \node [orange] at (9.3, 0) {\textbullet};
     \draw[black, very thick] (0.55,2) -- (0.65, 2);
          \draw[green, very thick] (0.55,2) -- (0.65, 2);
            \draw[black!30!green, very thick] (0.65,2) -- (0.65, 1.6);
                    \draw[black!60!green, very thick] (0.7, 1.6) -- (0.7, 1.4);
                                      \draw[black!60!green, very thick] (0.7, 1.4) -- (8.94, 1.4);
                                       \draw[black!60!green, very thick] (8.94, 0) -- (8.94, 1.4);
            \draw[black!30!green, very thick] (8.6,1.6) -- (0.65, 1.6);
             \draw[black!30!green, very thick] (8.6,1.6) -- (8.6, 0);
              \node [black] at (8, -0.3) {$p t$};
               \node [green!40!red] at (0.8, -0.3) {$\xi_1$};
                 \node [black] at (0.8, 0) {\textbullet};
                 \end{tikzpicture}
\caption{If $\tau - \epsilon < t < \tau$ then there are infinitely many pre-orbits (two in the case above -- in shades of green) that trace back to the interval from $pt$ to $p t + \delta$. }
\label{kksdf5}
\end{figure}
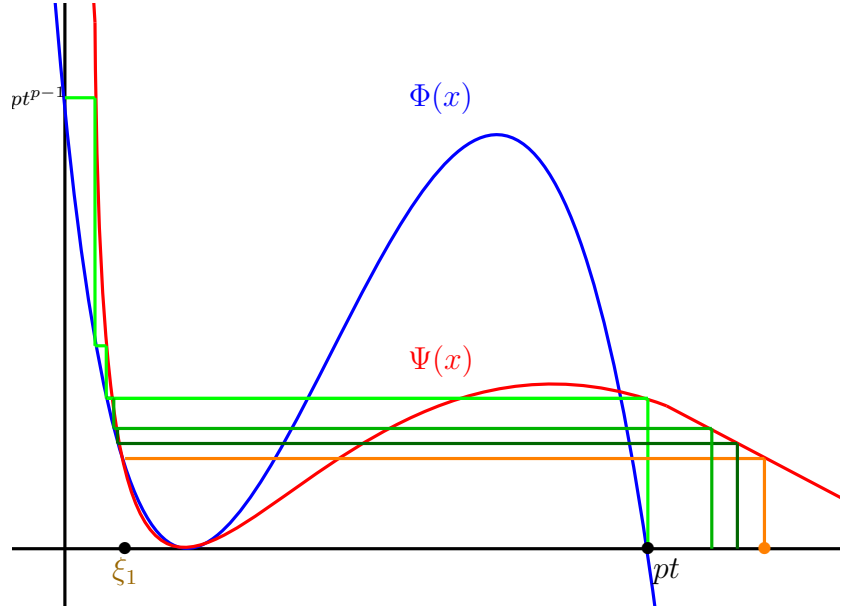

As $t$ continues to increase toward $\tau$, more and more orbit points are needed with abscissas within the interval $(\xi_2, pt)$ in order to escape toward the exit point.  This tells us that there are solutions to \eqref{system1}, but only for sufficiently large $d$.  Figure \ref{kksdf6} displays one such orbit.  More and more zigzags need to occur near the rightmost fixed point in order for the path to find its way to the exit point.

Thus for a sequence of values of $t < \tau_p$, and $d$ sufficiently large, we can find  solutions $R_1$, $R_2$,\ldots, $R_d$ to the system  \eqref{system1}.  Hence $T_{d,p} $ is at least as large as this value of $t$. 

On the other hand, it is clear that if $t \geq \tau_p$, then there are no orbits from $(pt, 0)$ that reach the exit point $(0, pt^{p-1})$.
This proves that $\lim_{d \rightarrow \infty} T_{d,p}  = \tau_p$. See Figure \ref{kksdf6}.
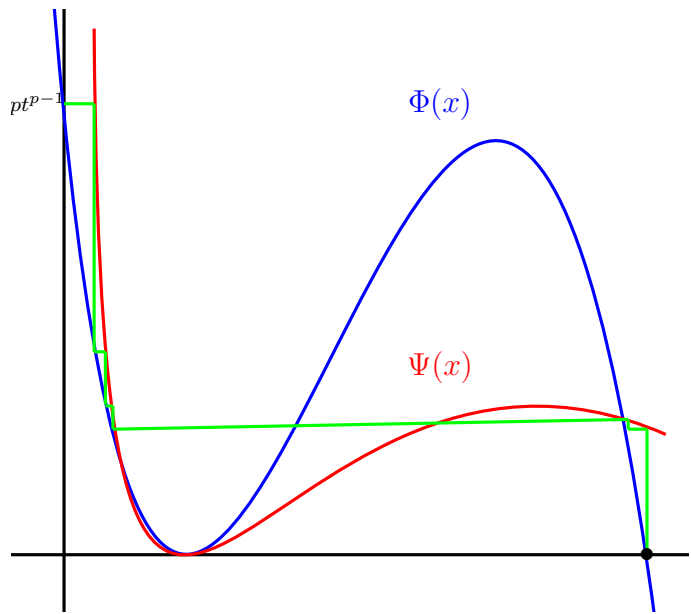
\begin{figure}[h]
\centering
\begin{tikzpicture}[scale = 1.0]
\clip (-0.7,-0.75) rectangle + (9,8);
\draw[very thick, black] (-1, 0) -- (11, 0);
\draw[very thick, black] (0, -2) -- (0, 11);
 \draw[very thick, blue] (-0.26, 9) .. controls (1.4, -15.3) and (5.4, 19.8) ..  (8, -2);
  \draw[very thick, red] (0.4, 7) .. controls (0.5, -6.4) and (2.5, 4) ..  (8, 1.6);
      \draw[green, very thick] (7.75, 0) -- (7.75, 1.67);
        \draw[green, very thick] (7.5, 1.67) -- (7.75, 1.67);
          \draw[green, very thick] (7.5, 1.67) -- (7.5, 1.8);
             \draw[green, very thick] (0.65, 1.67) -- (7.5, 1.8);
              \draw[green, very thick] (0.65, 1.67) -- (0.65, 1.98);
                  \draw[green, very thick] (0.55, 1.98) -- (0.65, 1.98);
       \node [black] at (7.75, 0) {\textbullet};
             \draw[green, very thick] (0.55,2) -- (0.55, 2.7);
              \draw[green, very thick] (0.4,2.7) -- (0.55, 2.7);
               \draw[green, very thick] (0.4,2.7) -- (0.4, 6);
               \draw[green, very thick] (0, 6) -- (0.4, 6);
                  \node [black] at (-0.35, 6) {{\tiny $p t^{p - 1}$}};
                    \node at (5,6) {${\color{blue}\Phi(x)}$};
    \node at (5,2.5) {${\color{red}\Psi(x)}$};
\end{tikzpicture}
\caption{When $t$ is less than but very close to $\tau$, numerous orbit points near the rightmost fixed point may be needed in order for the path to find a route toward the exit point.}
\label{kksdf6}
\end{figure}

It remains to show that $\Omega_p$ is open, which is to say that the value $t_f = \tau_p$ is not attained for any $f \in \ell^p_A$.  Such a function would still have to optimize the value of $t$ subject to the constraint $h'(t) = 0$.   The Lagrange multiplier theorem for Banach spaces applies (see, for example, \cite{Lue,Zei}), to the effect that the recurrence relations \eqref{recrelrk} must hold for all $k \in \mathbb{N}$.

 If $t = \tau_p$, however, then the orbit of $(pt, 0)$ converges toward the point $(\xi_2, \Psi(\xi_2))$.  Since $\xi_2 > 1$, this tells us that the corresponding coefficient ratios $R_1$, $R_2$, $R_3$,\ldots exceed $\xi_2$, and hence the coefficients of a solution to the recurrence relations 
\eqref{system1} must increase at least exponentially.   Hence there is not solution $f \in \ell^p_A$ for which $\tau_p$ is attained, and therefore $\Omega_p$ is an open region. See Figure \ref{kksdf8}.
\end{proof}

\begin{figure}[h]
\centering
\begin{tikzpicture}[scale = 1.0]
\clip (-0.5,-0.75) rectangle + (10,8);
\draw[very thick, black] (-1, 0) -- (11, 0);
\draw[very thick, black] (0, -2) -- (0, 11);
 \draw[very thick, blue] (-0.15, 9) .. controls (1.4, -15.3) and (5.4, 19.8) ..  (8, -2);
  \draw[very thick, red] (0.4, 7) .. controls (0.5, -6.4) and (2.5, 4) ..  (8, 1.6);

                    \node at (5,6) {${\color{blue}\Phi(x)}$};
    \node at (5,2.5) {${\color{red}\Psi(x)}$};
  \draw[green, dotted, very thick] (0.65, 1.82) -- (7.49,  1.82);
   \draw[green, very thick] (7.75, 1.7) -- (7.75, 0);
    \draw[green, very thick] (7.44, 1.7) -- (7.75, 1.7);
      \draw[green, very thick] (7.44, 1.7) -- (7.44, 1.8);
\end{tikzpicture}
\caption{The situation when $t = \tau_p$.}
\label{kksdf8}
\end{figure}
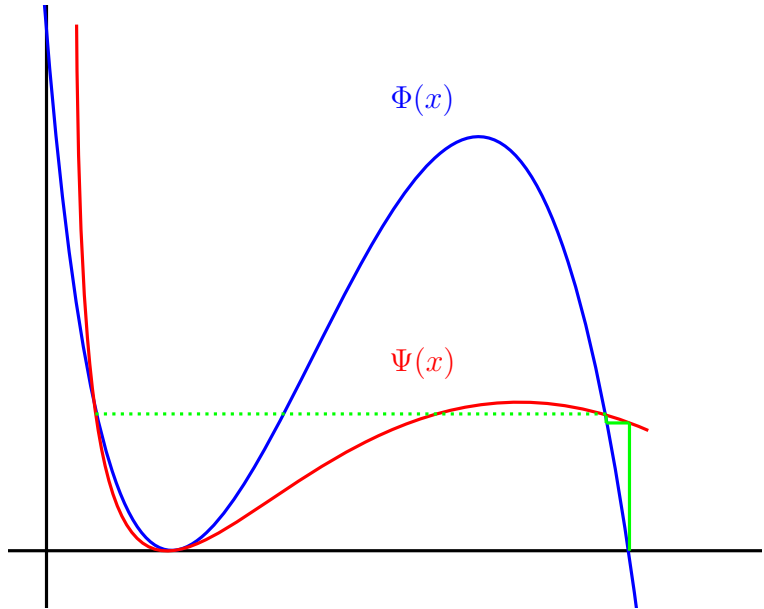

When $1<p<2$, the proof follows the same ideas steps, but with the graphs of $\Phi$ and $\Psi$ having a different character.  From Figure \ref{finaltikz}, we see that $\Phi$ and $\Psi$ have a vertical asymptote at $x=t$, rather than an intercept.  But as before, we are interested in orbits that originate from $(pt,0)$, and which exit at $(0,pt^{p-1})$, with the relative positions of the fixed points playing an important role.

It is challenging to calculate $\tau_p$.  Doing so requires finding that value $\tau_p$ of $t$ for which $\Phi(\xi_1) = \Phi(\xi_2)$, where $\xi_1$ and $\xi_2$ are the two distinct solutions for $x$ satisfying the condition
$
     x^{p-1}(pt-x) = p-1.
$
Here are some approximate values of $\tau_p$, obtained using Mathematica.  The formula for $\tau_p$ is very unstable, as it involves both extremely large numbers as well as extremely small ones; hence we looked only only at even integers for $p$, which made some simplifications possible.

$$\begin{tabular}{r|l}
       $p$ & $\tau_p \approx$ \\ \hline
       4 & 1.21157\\
       6 & 1.37386\\
       8 & 1.47757\\
       10 &  1.54974\\
       12 & 1.60310\\
       14 & 1.64431\\
       16 & 1.67719\\
       18 & 1.70408\\
       20 & 1.72654\\
\end{tabular}$$

\begin{figure}[h]
\centering
\begin{tikzpicture}[scale = 1.0]
\clip (-0.5,-0.75) rectangle + (10,8);
\draw[very thick, black] (-1, 0) -- (11, 0);
\draw[very thick, black] (0, -2) -- (0, 11);
\draw[very thick, dashed, black] (4, 0) -- (4, 11);
 \draw[very thick, blue] (-0.24, 6) .. controls (3, -7.37) and (4.5, 19.8) ..  (2.1, 9);
  \draw[very thick, blue] (4,9) .. controls (7, -2) and (8.5,0) ..  (11, -2);
  \draw[very thick, red] (0.38,5.4) .. controls (0.7, 0.8) and (2.2,0) ..  (4.5, 10);
    \draw[very thick, red] (4.5,6) .. controls (7, -1) and (6.5,6.6) ..  (30, -3);
 \draw[green, very thick] (0.7,3.06) -- (0.7, 3.2);
  \draw[green, very thick] (0.63,3.2) -- (0.7, 3.2);
    \draw[green, very thick] (0.63,3.2) -- (0.63, 3.5);
        \draw[green, very thick] (0.63,3.5) -- (0.5, 3.5);
                \draw[green, very thick] (0.5,3.5) -- (0.5, 4.1);
                               \draw[green, very thick] (0.5,4.1) -- (0.3, 4.1);
                                     \draw[green, very thick] (0.36,5.2) -- (0.3, 4.1);
                                         \draw[green, very thick] (0.36,5.2) -- (0, 5.2);
                   \draw[green, very thick] (0.7,3.06) -- (6.1, 3.06);
                    \draw[green, very thick] (7.8,2.547) -- (7.8, 0);
                     \draw[green, very thick] (6.19,2.547) -- (7.8, 2.547);
                        \draw[green, very thick] (6.19,2.547) -- (6.15, 3.06);
                           \node [black] at (4, -0.4) { $t$};
                            \node [black] at (4, 0) {\textbullet};
                            \node [black] at (7.5, -0.4) { $p t$};
                             \node [black] at (7.8, 0) {\textbullet};
                              \node [black] at (2, -0.4) { $(p - 1) t$};
                               \node [black] at (1.5, 0) {\textbullet};
                                \node at (7.5, 3.5) {${\color{red}\Psi(x) = \frac{(p - 1)}{x |1 - t/x|^{2 - p}}}$};
                                 \node at (7, 5) {${\color{blue}\Phi(x) = \frac{p t - x}{ |x - t|^{2 - p}}}$};
\end{tikzpicture}
\caption{A typical picture of $\Phi$, $\Psi$, and a successful orbit, when $1<p<2$.}
\label{finaltikz}
\end{figure}
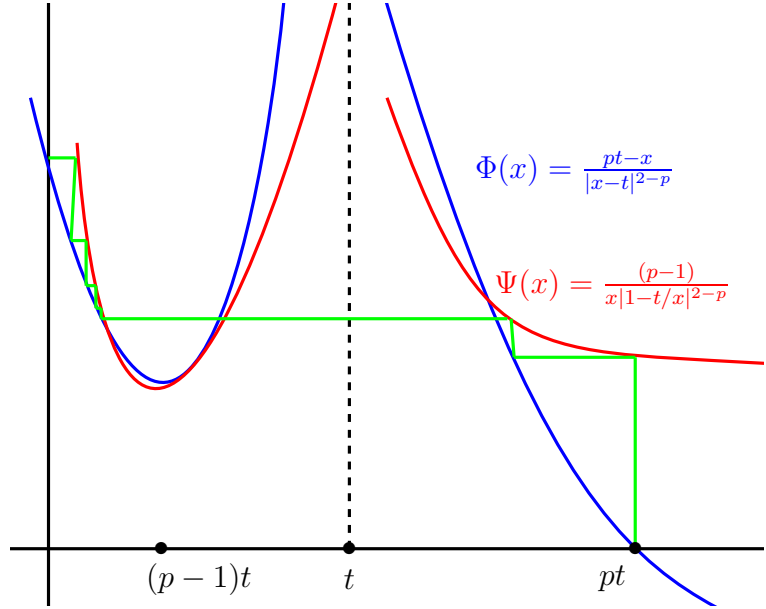

\section{Further directions}

The work in \cite{BenRMI} deals with more general Hilbert spaces. Most of what we presented here can be extended with minimal changes to the context of \emph{weighted} $\ell^p_A$ spaces (like the ones in \cite{SeTe}, that is, defined by a norm $\|f\|^p_{p,\omega} := \sum |a_k|^p \omega_k$ for some weight $\omega$) but the nonlinearity of the recurrence relations is a key obstruction then and the dynamical systems approach of the previous section is not directly feasible.

A few other minor complications are the following:
\begin{enumerate}
\item The fact that $\tau_p<2$ will not be preserved unless strong assumptions are made on the weights.
\item The examples to show that $\tau_p >1$ will need to be adapted but they will cover reasonable non-decreasing weights (doubling, perhaps with restrictions on their growth) and essentially all decreasing weights produce extra zeros for the function $1+ p \sum_{j=1}^k z^j$ for large $k$. In the weighted case for $p=2$, zeros of OPA intersect the disk if and only if the weight is decreasing. Once more, this establishes a clear dichotomy between the case $p=2$ and $p \neq 2$.
\end{enumerate}

Finally, it seems natural to extend this study to any other convex space where polynomials are dense, such as Hardy spaces $H^p$ or Bergman spaces $A^p$, for $1<p <\infty$. When $p=2$, both these spaces were dealt with in \cite{BenRMI} giving qualitatively different behaviors. The Hardy space $H^2$ corresponds to the case mentioned here of $\ell^2_A$, where no extra zeros are found, while $A^2$ produces a region playing the role of our $\Omega_p$ but with the difference of \emph{being closed} and the complement of a disk of radius strictly smaller than 1.
Partly, the case $p=2$ is special in that the recurrence relations \eqref{system1} become linear, and $a_{k+1}$ may be obtained directly from the values of $a_k$ and $a_{k-1}$. The approach in \cite{BenRMI} is based on the classical Favard's Theorem, which establishes a correspondence between solutions to linear recurrence relations, $L^2$ norms of Jacobi matrices and differential equations. This allows us to use the tools from several different mathematical worlds. The main obstruction in the non-Hilbert space case is that the recurrence relations  \eqref{system1} become non-linear and Favard's Theorem is not applicable: the solutions to each of the recurrence relations is non-unique and can only be described implicitly.
 At least for bounded functions, one can expect the zero of their linear approximants to move with some form of continuity when varying the parameter $p$ defining the space $H^p$ or $A^p$. Thus, extra zeros could appear in $A^p$ at least for all $p$ in some small region around $p=2$, while one should, in principle, expect $H^p$ to present no extra zeros. The natural candidate for a function producing extra zeros in $A^p$ is
$$f(z) = \left(1- \frac{z}{\sqrt{2}}\right)^{-3},$$
 the extremizer found in \cite{BenRMI}.

\bigskip

\noindent\textbf{Acknowledgements.} Seco acknowledges financial support by the Spanish Ministry of Economy and Competitiveness, through the ``Severo Ochoa Programme for Centers of Excellence in R\&D'' (CEX2019-000904-S) and through grant PID2019-106433GB-I00; and by the Madrid Government (Comunidad de Madrid-Spain) under the Multiannual Agreement with UC3M in the line of Excellence of University Professors (EPUC3M23), and in the context of the V PRICIT (Regional Programme of Research and Technological Innovation).

\bibliographystyle{plain}

\bibliography{OptimalApprox}

\begin{thebibliography}{10}

\bibitem{MR1878629}
Evgeny Abakumov and Alexander Borichev.
\newblock Shift invariant subspaces with arbitrary indices in {$l^p$} spaces.
\newblock {\em J. Funct. Anal.}, 188(1):1--26, 2002.

\bibitem{MR1440934}
A.~Aleman, S.~Richter, and C.~Sundberg.
\newblock Beurling's theorem for the {B}ergman space.
\newblock {\em Acta Math.}, 177(2):275--310, 1996.

\bibitem{RemBen}
Catherine B\'{e}n\'{e}teau, Matthew~C. Fleeman, Dmitry~S. Khavinson, Daniel
  Seco, and Alan~A. Sola.
\newblock Remarks on inner functions and optimal approximants.
\newblock {\em Canad. Math. Bull.}, 61(4):704--716, 2018.

\bibitem{BenRMI}
Catherine B\'{e}n\'{e}teau, Dmitry Khavinson, Constanze Liaw, Daniel Seco, and
  Brian Simanek.
\newblock Zeros of optimal polynomial approximants: {J}acobi matrices and
  {J}entzsch-type theorems.
\newblock {\em Rev. Mat. Iberoam.}, 35(2):607--642, 2019.

\bibitem{BenJLMS}
Catherine B\'{e}n\'{e}teau, Dmitry Khavinson, Constanze Liaw, Daniel Seco, and
  Alan~A. Sola.
\newblock Orthogonal polynomials, reproducing kernels, and zeros of optimal
  approximants.
\newblock {\em J. Lond. Math. Soc. (2)}, 94(3):726--746, 2016.

\bibitem{MR27954}
Arne Beurling.
\newblock On two problems concerning linear transformations in {H}ilbert space.
\newblock {\em Acta Math.}, 81:239--255, 1948.

\bibitem{ChengRoss15}
R.~Cheng and W.~T. Ross.
\newblock Weak parallelogram laws on {B}anach spaces and applications to
  prediction.
\newblock {\em Period. Math. Hungar.}, 71(1):45--58, 2015.

\bibitem{CD}
Raymond Cheng and James~G. Dragas.
\newblock On the failure of canonical factorization in
  $\ell^p_{\uppercase{a}}$.
\newblock {\em J. Math. Anal. Appl.}, 479:1939--1955, 2019.

\bibitem{CMR}
Raymond Cheng, Javad Mashreghi, and William~T. Ross.
\newblock {\em Function theory and $\ell^p$ spaces}, volume~75 of {\em
  University Lecture Series}.
\newblock American Mathematical Society.

\bibitem{Chengetal1}
Raymond Cheng, Javad Mashreghi, and William~T. Ross.
\newblock Inner functions and zero sets for {$\ell^p_A$}.
\newblock {\em Trans. Amer. Math. Soc.}, 372(3):2045--2072, 2019.

\bibitem{Chengetal2}
Raymond Cheng, Javad Mashreghi, and William~T. Ross.
\newblock Inner functions in reproducing kernel spaces.
\newblock In {\em Analysis of operators on function spaces}, Trends Math.,
  pages 167--211. Birkh\"{a}user/Springer, Cham, 2019.

\bibitem{Dur}
Peter~L. Duren.
\newblock {\em Theory of {$H^{p}$} spaces}.
\newblock Pure and Applied Mathematics, Vol. 38. Academic Press, New
  York-London, 1970.

\bibitem{Gar}
John~B. Garnett.
\newblock {\em Bounded analytic functions}, volume~96 of {\em Pure and Applied
  Mathematics}.
\newblock Academic Press, Inc. [Harcourt Brace Jovanovich, Publishers], New
  York-London, 1981.

\bibitem{Lue}
David~G. Luenberger.
\newblock Local theory of constrained optimization.
\newblock In {\em Optimization by Vector Space Methods}, pages 239--270. John
  Wiley \& Sons, New York, 1969.

\bibitem{MR1259923}
Stefan Richter and Carl Sundberg.
\newblock Multipliers and invariant subspaces in the {D}irichlet space.
\newblock {\em J. Operator Theory}, 28(1):167--186, 1992.

\bibitem{SeTe}
Daniel Seco and Roberto T\'{e}llez.
\newblock Polynomial approach to cyclicity for weighted {$\ell^p_A$}.
\newblock {\em Banach J. Math. Anal.}, 15(1):Paper No. 1, 16, 2021.

\bibitem{Zei}
Eberhard Zeidler.
\newblock {\em Applied Functional Analysis: Variational Methods and
  Optimization}.
\newblock Applied Mathematical Sciences 109. Springer-Verlag, New York, 1995.

\end{thebibliography}

\end{document}